\documentclass[]{siam_preamble/siamart251216}

\usepackage[utf8]{inputenc}
\usepackage{amsmath,amsfonts,amssymb}

\usepackage[dvipsnames]{xcolor}
\usepackage{graphicx}
\usepackage{svg}

\usepackage{yhmath}

\usepackage{pifont}

\usepackage[capitalize]{cleveref}
\crefname{section}{Section}{SectionsS}

\usepackage{enumitem}

\usepackage{tikz-cd}

\newcommand{\sspan}{\operatorname{span}}

\newcommand{\diag}{\operatorname{diag}}

\newcommand{\norm}[1]{\left\lVert#1\right\rVert}
\newcommand{\p}{\partial}

\newcommand{\vertiii}[1]{{\left\vert\kern-0.25ex\left\vert\kern-0.25ex\left\vert #1 
    \right\vert\kern-0.25ex\right\vert\kern-0.25ex\right\vert}}

\renewcommand{\P}{\mathcal{P}}
\newcommand{\G}{\mathcal{G}}
\newcommand{\M}{\mathcal{M}}

\newcommand{\hh}{\hat{h}}
\newcommand{\hu}{\hat{u}}
\newcommand{\huu}{\widehat{\u}}

\newcommand{\hff}{\widehat{\f}}
\newcommand{\hp}{\hat{p}}
\newcommand{\hq}{\hat{q}}
\newcommand{\hr}{\hat{r}}
\newcommand{\hrho}{\hat{\rho}}
\newcommand{\hE}{\hat{E}}
\newcommand{\hm}{\hat{m}}

\newcommand{\what}{\widehat}

\renewcommand{\u}{\mathbf{u}}
\newcommand{\f}{\mathbf{f}}

\newcommand{\x}{\mathbf{x}}

\newcommand{\zero}{\mathbf{0}}

\newcommand{\sR}{\mathsf{R}}
\newcommand{\sM}{\mathsf{M}}
\newcommand{\sE}{\mathsf{E}}
\newcommand{\sC}{\mathsf{C}}

\newcommand{\sU}{\mathsf{U}}
\newcommand{\sP}{\mathsf{P}}
\newcommand{\sH}{\mathsf{H}}

\newcommand{\R}{\mathbb{R}}

\newcommand{\PK}{\Phi_K}
\newcommand{\ps}{\underset{\mathrm{ps}}{\ast}}
\newcommand{\co}{\underset{\mathrm{co}}{\ast}}

\newcommand{\as}{\underset{\mathrm{as}}{\ast}}
\newcommand{\pinf}{\underset{\infty}{\ast}}
\newcommand{\pzer}{\underset{0}{\ast}}

\newcommand{\fr}{\mathsf{fr}}
\newcommand{\hyp}{\operatorname{hyp}}

\newcommand{\cmark}{\ding{51}}%
\newcommand{\xmark}{\ding{55}}%

\newsiamremark{remark}{Remark}

\headers{Hyperbolicity-Preserving Stochastic Galerkin}{H. Meghaichi and Y. Xing}

\title{
Hyperbolicity-Preserving Stochastic Galerkin Methods for Conservation Laws
Based on Associative Truncated Products on Polynomial Spaces
}
\author{Haroun Meghaichi \and Yulong Xing\thanks{Department of Mathematics, The Ohio State University, Columbus, OH 43210, USA. {\tt meghaichi.1@osu.edu, xing.205@osu.edu}}
}

\begin{document}

\maketitle

\begin{abstract}
Stochastic Galerkin discretizations of nonlinear hyperbolic
conservation laws may lose hyperbolicity because the standard
pseudospectral product is generally nonassociative, leading to
non-commuting blocks in the flux Jacobian matrix. We develop a novel framework for
constructing hyperbolicity-preserving stochastic Galerkin systems based on
associative truncated products on polynomial spaces.
In one stochastic dimension, we characterize associative
truncated products through a single polynomial datum and identify examples with
useful symmetry, positivity, and spectral properties, including collocation
products and an associative symmetric product based on Gaussian quadrature nodes. 
We prove a consistency result showing that, under suitable 
projection-error assumptions, these products converge to the classical product
as the polynomial degree grows.
For systems with rational fluxes, we derive sufficient
conditions under which the resulting stochastic Galerkin flux remains
hyperbolic on the corresponding admissible set.
Applications to the one-dimensional isothermal and compressible Euler
equations show accurate statistical approximation and robust hyperbolicity
preservation of the computed stochastic Galerkin states.
\end{abstract}



\section{Introduction}

\label{sec:introduction}
Uncertainty quantification for hyperbolic conservation
laws has recently received considerable attention due to its
importance in many applications in engineering and physics, where 
initial conditions, boundary conditions, or model parameters are often uncertain. 
In such settings, the prediction made
from a deterministic model provides only one possible realization of the solution. 
The uncertainty quantification then aims to provide a statistical description of the solution ensemble to the non-deterministic model.
Examples of hyperbolic conservation laws with uncertain data include 
the shallow water equations, the isothermal Euler equations, and the general compressible Euler equations. 

Two main approaches have been commonly used to discretize the uncertainty in the
literature. \textit{Non-intrusive} methods, such as
the stochastic collocation method
\cite{
babuškaFiniteElementMethod1970,
chertockSplineBasedStochasticCollocation2025,
grzelakStochasticCollocationMonte2019,
gunzburgerSparseCollocationMethods2017,
nobileSparseGridStochastic2008,
xiuHighOrderCollocationMethods2005} and Monte Carlo sampling 
\cite{barthMultilevelMonteCarlo2013,
feireislMonteCarloMethod2026,mishraSparseTensorMultilevel2012a,
pareschiIntroductionMonteCarlo2001},
solve the deterministic model at specific selected or random points in
the stochastic space, and then reconstruct the
statistical moments of the solution from these samples.  
These approaches are non-intrusive since they can use existing deterministic
solvers with little modification, which makes them easy to implement and typically
preserves the properties of the deterministic system. 
Moreover, the simulation of the samples can be done
in parallel to reduce the wall-clock time of the simulation.

The second type of approach, which is the focus of this article, is the
\textit{intrusive} stochastic
Galerkin (SG) method, which relies on expanding the solution in terms of
orthogonal polynomials in the stochastic space, and projecting the governing
equations onto the finite-dimensional space spanned by these polynomials.
This approach has been widely used in the literature for various hyperbolic
systems of conservation laws with great success
\cite{
benderEntropyConservativeDiscontinuousGalerkin2024,
chertockChallengesStochasticGalerkin2026,
daiHyperbolicitypreservingWellbalancedStochastic2021,
daiHyperbolicitypreservingWellbalancedStochastic2022,
desprésRobustUncertaintyPropagation2013,
gersterHaartypeStochasticGalerkin2025,
lemaîtreSpectralMethodsUncertainty2010,
petterssonPolynomialChaosMethods2015,
schlachterHyperbolicitypreservingStochasticGalerkin2018}.
The SG method typically leads to high-order
convergence provided that the solution is smooth in the stochastic parameter(s),
and can be efficient if accurate statistical information are required.
However, it is well known (e.g.,
\cite{jinStudyHyperbolicityKinetic2019,desprésRobustUncertaintyPropagation2013})
that the SG method may lead to a non-hyperbolic deterministic system
in certain situations unless special care is taken, which may lead to issues
with the well posedness of the system and the stability of the numerical method
used. Multiple approaches have been proposed in the literature to overcome this
issue, such as 
transforming the SG method to a stochastic collocation method at the quadrature nodes 
\cite{zhongEntropyStableGalerkin2022},
using the Roe-variable transformation
\cite{petterssonStochasticGalerkinMethod2014},
the entropy-based approach
\cite{poëtteUncertaintyQuantificationSystems2009},  the Haar wavelet basis
instead of the polynomial chaos basis
\cite{gersterHaartypeStochasticGalerkin2025,petterssonStochasticGalerkinMethod2014},
and problem-specific modifications of
the SG system, for example for the shallow water
equations \cite{daiHyperbolicitypreservingWellbalancedStochastic2021}. 
These approaches are effective in important cases, but their extension to broader classes of
nonlinear conservation laws is not always straightforward. 

In this work, we propose a new framework based on the algebraic structure of the 
stochastic polynomial space, to design hyperbolicity-preserving
SG methods for hyperbolic conservation laws. 
The standard pseudospectral product, also referred to as the Galerkin product, is 
obtained by projecting pointwise products back
onto the stochastic approximation space.
It is well-known that the pseudospectral product is not
associative in general (e.g., in
\cite[p. 255]{sullivanIntroductionUncertaintyQuantification2015}), 
and this lack of associativity is closely related to the loss of
hyperbolicity in the SG method since it leads to Jacobian bloc matrices where
simultaneous diagonalization is not available. 
One idea is to use a space on which the pseudospectral product 
is associative,
as discussed in
\cite{petterssonStochasticGalerkinMethod2014} for Euler's equations and in
\cite{gersterHaartypeStochasticGalerkin2025} for a wider class of systems, where
the discretization in the stochastic space is carried out using a Haar-type basis.
Our main idea is to use an associative truncated polynomial product 
on the same standard polynomial space to define the discrete multiplication
operator in the SG method. This preserves the high-order approximation structure of
polynomial chaos expansions, and at the same time 
restores the algebraic associativity  needed for hyperbolicity arguments.
In particular, we study a class of truncated products, that form a unital algebra
on the polynomial space with the added condition that the truncated product
coincides with the classical product whenever possible.

The main contributions of the paper are as follows.
\begin{enumerate}[label=(\roman*)]
\item We introduce truncated products on stochastic polynomial spaces and
identify algebraic properties, including symmetry, associativity,
hyperbolicity, uniform hyperbolicity, and positive definiteness, that are
relevant for SG discretizations.
\item In one stochastic dimension, we characterize associative truncated
products on the $P_K$ polynomial space through the single element ($\phi_1\ast\phi_K$). We show
that symmetric associative products are determined by one scalar parameter and
derive spectral criteria for hyperbolicity and positivity. 
Extension to multidimensional random variables via tensor products are also discussed.
\item We show that the pseudospectral product is not associative on
$P_K$ for $K\ge2$, and we construct several associative alternatives,
including collocation products and an associative symmetric product based on
Gaussian quadrature nodes. The latter shares the first multiplication matrix
with the pseudospectral product but constructs the higher multiplication
matrices differently so as to enforce associativity.
\item We prove an approximation result showing that, under suitable conditions
on the projection errors and the operator norms of the truncated products, the
truncated products converge to the usual product as the polynomial degree
increases.
\item For conservation laws with rational fluxes, we derive sufficient
conditions under which the SG flux obtained from an associative truncated
product remains hyperbolic on the corresponding admissible set. We apply this
framework to the isothermal Euler equations and the one-dimensional compressible
Euler equations.
\end{enumerate}
Numerical experiments demonstrate that the proposed SG formulations remain
hyperbolic in the tested regimes and accurately capture statistical quantities
such as the mean, variance, median, and confidence intervals. In smooth
settings, the observed convergence agrees with the expected order of the
finite-volume discretization, while it captures the behavior
of the method for nonsmooth solutions in shock-tube examples.

This paper is organized as follows. In \cref{sec:notation}, we introduce
the notation used throughout the article. In \cref{sec:truncated_product}, we
discuss the definition and properties  of truncated products on polynomial
spaces, as well as examples and important results such that the approximation
property and tensor-product extensions. 
In \cref{sec:hyperbolic_systems}, we derive the hyperbolicity-preserving
criterion for SG systems using the truncated products, and apply it to 
the isothermal Euler and compressible Euler equations.
Finally, we present numerical experiments in \cref{sec:numerical_results} 
and concluding remarks in \cref{sec:conclusion}.

\section{Notation}
\label{sec:notation}
Throughout the paper, let $(\Xi,\mathcal{F},\mu)$ be a probability
space with $\Xi\subset{\mathbb{R}}^{d}$, and $w$ be a density function on
$\Xi$.  We use
$\langle\cdot,\cdot\rangle_{w}$ and $\norm{}_{w,\Xi}$ to denote the usual
weighted inner product and norm on $L^{2}_{w}(\Xi)$, namely
\begin{subequations}
\begin{eqnarray}
\langle f,g\rangle_{w} = \int_{\Xi} fg\  \mathrm{d}\mu 
= \int_{\Xi} f(\xi)g(\xi)\ w(\xi) \mathrm{d}\xi,\qquad
\norm{f}_{w,\Xi}=\sqrt{\langle f,f\rangle_{w}},\\
L^2_{w}(\Xi)=\{f:\Xi\to \mathbb{R}, \text{ measurable}\mid
\norm{f}_{w,\Xi}<\infty\}.
\end{eqnarray}
\end{subequations}
Let $\{\phi_k\}_{k=0}^{\infty}$ be a complete orthonormal family of $L^2_{w}(\Xi)$ with $\phi_0\equiv 1$, and
we use $\Phi_K$ to denote $\sspan \{\phi_k\}_{k=0}^{K}$ where $K\ge
0$. In particular, we are interested in orthonormal polynomial bases on
$\Xi$, and we use $P_K(\Xi)$ to denote the space of polynomials of
total degree at most $K$ on $\Xi$, and $Q_K(\Xi)$ to denote the space
of tensor-product polynomials of degree at most $K$ in each variable.

Given a domain $\Omega\subset \mathbb{R}^{m}$  and a function
$h:\Omega\times [0,T]\times \Xi \to \mathbb{R}$, with $h(x,t,\cdot)\in L^2_{w}(\Xi)$ for all $(x,t)\in \Omega\times
[0,T]$, we denote its generalized polynomial chaos coefficients by $\hh=(\hh_k)_{k=0}^{\infty}$ with
$\hh_k(x,t)=\langle h(x,t,\cdot), \phi_{k}(\cdot) \rangle_{w}$. 
We use a similar notation for functions in the finite-dimensional
space $\PK$, where $\hp_k$ denotes the coefficient of $p\in\PK$ in
the basis $\{\phi_k\}_{k=0}^K$, and $\hp=(\hp_0\cdots,\hp_K)^T$ is its coefficient
vector, that is,
\begin{equation}p(\xi)=\hp_0 \phi_0(\xi)
+ \hp_1 \phi_1(\xi)+\dots+\hp_K \phi_K(\xi)
=\sum_{k=0}^K \hp_k \phi_k(\xi).
\label{eqn:p_expansion}
\end{equation}
For simplicity, we will denote by $\G_K$ the reconstruction operator that maps $\hp$ to $p$ defined in
\eqref{eqn:p_expansion}, and $\Pi_K$ the orthogonal
 projection from $L^2_{w}(\Xi)$ onto $\PK$.

\section{Truncated product spaces}
\label{sec:truncated_product}
The polynomial space $\PK$ is typically not closed under binary and unary operations
such as multiplication, division or square root. Thus, intrusive SG discretizations of
the nonlinear terms in PDEs typically require a discrete
analog of these operations on $\PK$. In principle, a discrete analog of
division, square root, and other operations can be defined directly from a given
discrete multiplication operator as discussed in
\cite{debusschereNumericalChallengesUse2004}. In this paper,
we focus on multiplication operators that preserve certain
properties of the usual
multiplication, which we will describe in the following definition.

\begin{definition}\label{def:truncated_product}
A binary operation $\ast: \PK^2 \to \PK$ is called a
\textit{truncated product} on $\PK$ if it satisfies the following conditions
\begin{enumerate}[label=(\alph*)]
\item Distributivity: The map $(p,q)\to p\ast q$ is bilinear.
\item Commutativity: For all $p,q\in \PK$ we have $p\ast q=q\ast p$.
\item Consistency: If $pq\in \PK$, then $p\ast q=pq$.
\end{enumerate}
If these conditions are met, we say that $(\PK,\ast)$ is a truncated product space.
\end{definition}

A truncated product is fully determined by its action on the basis $\{\phi_{k}\}_{k=1}^K$, 
i.e., the products $\phi_k\ast\phi_\ell$, $1\le k,\ell\le K$.
Hence, the quantities $\langle \phi_i, \phi_k\ast
\phi_\ell\rangle_{w}$ fully describe the truncated product $\ast$, and
we will refer to them as the structure constants following the
nomenclature in Lie Algebra where they appear (although in a
different context, \textit{e.g.}, \cite{cornwellGroupTheoryPhysics1997}). More
precisely, for $p,q\in \PK$ with coefficient vectors $\hp,\hq\in\mathbb{R}^{K+1}$, we have
\begin{equation}
\widehat{p\ast q}= \P(\hp)\hq,\ \ \P(\hp)=\sum_{k=0}^{K}\hp_k
\M^k,\ \
\M^{k}_{i+1,j+1}:= \langle \phi_i, \phi_k\ast
\phi_j\rangle_{w},\ 0\le i,j,k\le K.
\label{eqn:pq_M_form}
\end{equation}
We call $\P:\mathbb{R}^{K+1}\to
\mathbb{R}^{(K+1)\times(K+1)}$ the multiplication
tensor. From a computational point of view, the truncated product can be
represented by it structure coefficients, and the matrices
$\{\M^k\}_{k=0}^K$ can be assembled offline for any given truncated product on $\PK$.

In addition to the conditions stated in \cref{def:truncated_product}, we
introduce additional properties that some truncated products may satisfy. These
properties will be useful in the analysis of the stochastic Galerkin method for
conservation laws.

\begin{definition}
Let $(\PK,\ast)$ be a truncated product space, we say that
\begin{itemize}
\item $(\PK,\ast)$ is symmetric iff $ \langle p, q\ast r\rangle_{w}=\langle
p\ast q,  r\rangle_{w}$ for all $p,q,r\in \PK$.
\item  $(\PK,\ast)$ is associative iff $(p\ast q)\ast r=p\ast
(q\ast r)$ for all $p,q,r\in \PK$.
\item   $(\PK,\ast)$ is hyperbolic if $\P(\hp)$ is diagonalizable with real eigenvalues for all $\hp\in \mathbb{R}^{K+1}$. 
\item  $(\PK,\ast)$ is uniformly hyperbolic if there exists an
invertible matrix $V$ such that $V^{-1}\P(\hp)V$ is diagonal for all $\hp\in
\mathbb{R}^{K+1}$.
\item $(\PK,\ast)$ is positive definite if all eigenvalues of $\P(\hp)$ are positive
whenever $p(\xi)>0$ for all $\xi\in\Xi$, where $p=\G_K(\hp)$.
\end{itemize}
\end{definition}

It is worth noting here that the first two definitions above can be
restated in terms of the matrices $\{\M^k\}_{k=0}^K$ as follows:

\begin{itemize}
\item $\ast$ is symmetric iff $(\M^k)^T=\M^k$ for all $0\le k \le K$.
\item $\ast$ is associative iff the matrices $\{\M^k\}_{k=0}^K$
commute, meaning that $\M^i\M^j=\M^j\M^i$ for all $0\le i,j\le K$.
\end{itemize}

Furthermore, the associativity of the truncated product implies  the following
important property of the multiplication tensor $\P$. The proof is provided in the appendix.
\begin{lemma}
\label{lem:associativity_PpPq}
Let $\ast$ be an associative truncated product, then
\begin{equation}
\P(\P(\hp)\hq)=\P(\hp)\P(\hq),\qquad \forall \hp,\hq\in\mathbb{R}^{K+1}.
\label{eqn:P_P_pq_first}
\end{equation}
Moreover, if $\P(\hp)$ is invertible, then
\begin{equation}
\P(\P(\hp)^{-1}\hq)=\P(\hp)^{-1}\P(\hq).
\label{eqn:P_P_pq_second}
\end{equation}
\end{lemma}

We also note that several direct connections among these
properties. For instance, every symmetric truncated product is
hyperbolic, and an associative  hyperbolic
truncated product is
uniformly hyperbolic, since the commuting family $\{\M^k\}_{k=0}^K$ is simultaneously
diagonalizable. 
Also, some of these definitions have previously appeared in the literature, although
under different names or in more restricted settings. For example, uniformly
hyperbolic truncated products are discussed in
\cite{gersterHaartypeStochasticGalerkin2025,petterssonStochasticGalerkinMethod2014}
among other works, where it is described as the case where
$\P(\hp)$ has constant eigenvectors (independent of $\hp$). The basis is
called an $\mathcal{A}$gPC basis as described in
\cite{gersterEntropiesSymmetrizationHyperbolic2020}.

\subsection{The lack of associativity of the pseudospectral product}
\label{subsec:ps_product}

One of the most widely used truncated products in Stochastic
Galerkin methods is the so-called pseudospectral product
\cite{benderEntropyConservativeDiscontinuousGalerkin2024,
daiHyperbolicitypreservingWellbalancedStochastic2021,
daiHyperbolicitypreservingWellbalancedStochastic2022,
debusschereNumericalChallengesUse2004,
gersterHyperbolicStochasticGalerkin2019,
jinStudyHyperbolicityKinetic2019,
lemaı̂treUncertaintyPropagationUsing2004,
xiuEfficientStochasticGalerkin2009},
which we denote by $\ps$. For two functions $p,q\in \PK$, it is defined as
\begin{equation}
p\ps q := \Pi_K(pq) =\sum_{k=0}^K c_k \phi_k,\qquad c_k=\langle
pq,\phi_k\rangle_{w},\quad k=0,1,\dots,K.
\label{eqn:def_ps}
\end{equation}
By construction, $p\ps q$ is the best approximation to
$pq$ in the norm $\norm{\cdot}_{w,\Xi}$ among all elements of $\PK$. It follows immediately from
the definition that the pseudospectral product is
a symmetric truncated product in the sense of
\cref{def:truncated_product}, and it is relatively easy to show
that this truncated product is positive definite (see, for instance, \cite[Theorem 3.2]{daiHyperbolicitypreservingWellbalancedStochastic2021}).
However, the pseudospectral product is not associative on polynomial spaces as the following lemma shows. 
\begin{lemma}
\label{lem:ps_not_associative}
Let $\Xi\subset \mathbb{R}$ and let $\PK=P_K(\Xi)$ with $K\ge 2$.
Then $\ps$ is not associative on $\PK$.
\end{lemma}
A proof is given in \cref{sec:appx_proofs}. This lack of associativity of the 
pseudospectral product is already known in the literature
\cite{debusschereNumericalChallengesUse2004,
gersterStabilizationUncertaintyQuantification2020,
sullivanIntroductionUncertaintyQuantification2015},
and it is central to the issue observed in
\cite{jinStudyHyperbolicityKinetic2019} when studying the
hyperbolicity of a certain SG formulation for the isothermal Euler
equations, among other conservation laws such as the isentropic
Euler equations and the general Euler equations. However, it is worth noting 
that the pseudospectral product can be associative
on other finite-dimensional spaces.
For instance, the Haar wavelet basis functions \cite{haarZurTheorieOrthogonalen1910} are
used the Stochastic Galerkin community for at least a decade with
great success
\cite{benderEntropyConservativeDiscontinuousGalerkin2024,
gersterHyperbolicStochasticGalerkin2019,
gersterHaartypeStochasticGalerkin2025,
petterssonStochasticGalerkinMethod2014}.

\subsection{Associative truncated products on
\texorpdfstring{$P_K$}{PK} in one dimension}
\label{subsec:associative_truncated_products_d1}

In this section, we propose an alternative
approach by studying the general class of associative truncated
products (ATPs). We now restrict to the one-dimensional setting
$\Xi\subset\mathbb{R}$ and $\PK=P_K(\Xi)$. In this setting, the orthonormal basis
$\{\phi_k\}_{k=0}^K$
satisfies a three-term recurrence relation of the form
\cite{szegőOrthogonalPolynomials1939}
\begin{equation}
\label{eqn:three_term_recurrence}
\phi_{k}=(a_k \phi_1 +b_k)\phi_{k-1}+c_k \phi_{k-2},\quad k\ge
2.
\end{equation}
The most important observation in this section is
that, in one dimension, every ATP is fully determined
by the single element $\phi_1\ast \phi_K$. 
Furthermore, an associative and symmetric truncated product
is fully determined by one scalar parameter, and the 
hyperbolicity and positivity properties of an ATP can be reduced 
to corresponding properties of the matrix $\M^1$.

\begin{theorem}
\label{thm:truncated_product_char}
Let $\Xi\subset \mathbb{R}$ and let $\ast$ be an ATP  
on $P_K(\Xi)$, then $\ast$ is uniquely determined by the element
$\phi_1\ast \phi_K.$
\end{theorem} 
\begin{proof}
From the consistency condition and our assumption,
we know that $\phi_1 \ast p=\phi_1 p$  for all $p\in P_{K-1}(\Xi)$. Now,
let us assume (for the sake of strong induction) that $\phi_k\ast p$
is well-defined for all $p\in P_{K}(\Xi)$ for all $k< k_0\le K$. 

Let $p\in P_K(\Xi)$, using the recurrence \eqref{eqn:three_term_recurrence} and the associativity, commutativity, distributivity properties of $\ast$, we have
\[
p\ast \phi_{k_0}=a_{k_0}\phi_1\ast \left(
p\ast \phi_{k_0-1}
\right)+b_{k_0}p\ast \phi_{k_0-1}+ c_{k_0}p \ast \phi_{k_0-2}.
\]
Hence, $p\ast \phi_{k_0}$ is well-defined. Since this holds for all
$k_0=2,\dots K$. We conclude that $p\ast \phi_k$ is well-defined for all $0\le
k\le K$.
Lastly, given $p,q\in P_K(\Xi)$, we can determine $p\ast q$ by expressing $q$
in terms of the basis $\{\phi_k\}_{k=0}^K$ and using distributivity.
\end{proof}

In other terms, the matrices $\{\M^k\}_{k=2}^K$ can be calculated directly from
$\M^0=I$ and $\M^1$, using the recurrence relation
\eqref{eqn:three_term_recurrence} in
the following way
\begin{equation}
\M^k=a_k \M^1\M^{k-1}+b_k\M^{k-1}+c_k\M^{k-2},\qquad k=2,3,\dots K.
\label{eqn:Mk_rec_relation}
\end{equation}
This observation allows us to deduce the following two lemmas.

\begin{lemma}
\label{lem:symmetry_of_M_k}
Let $\Xi\subset \R$ and let $\ast$ be an ATP on
$P_K(\Xi)$. Let
$\{\M^k\}_{k=0}^K$ be the associated matrices as defined in
\eqref{eqn:pq_M_form}. If $\M^1$ is symmetric, then $\M^k$ is symmetric for all
$0\le k\le K$.
\end{lemma}

\cref{thm:truncated_product_char} reveals that an
ATP is determined by $K+1$ coefficients of $\phi_1\ast\phi_K$, which correspond to the last
column of $\M^1$. If the ATP is assumed to be symmetric, the
number of parameters defining the product is reduced to a single scalar. 

\begin{lemma}\label{lem:uniqueness_AS_one_d}
Let $\Xi\subset \mathbb{R}$ and let $\ast$ be a symmetric and
associative product on $P_K(\Xi)$, then $\ast$ is uniquely determined by the scalar 
$
\langle\phi_1\ast \phi_K,\phi_K \rangle_{w}.
$
\end{lemma}
\begin{proof}
From the symmetry property of $\ast$, we know that
$\langle \phi_1\ast \phi_K,\phi_j\rangle_{w}=
\langle\phi_1\ast \phi_j,\phi_{K}\rangle_w,
$
which is already determined by the consistency property for all $0\le j\le K-1$.
This leaves only $\langle \phi_1\ast \phi_K,\phi_K\rangle$ as a free
parameter in
$\phi_1\ast \phi_K$. Hence, this scalar quantity determines $\M^1$, and
consequently determines $\ast$ by \cref{thm:truncated_product_char}.
Furthermore, \cref{lem:symmetry_of_M_k} guarantees that the matrices
$\{\M^k\}_{k=1}^{K}$ is symmetric
if $\M^1$ is symmetric. Hence, the ATP generated
from the three term recursion is symmetric.
\end{proof}

\begin{lemma}\label{lem:Results_about_ATP}
Let $\Xi\subset \mathbb{R}$ and let $\ast$ be an ATP on $P_K(\Xi)$. Then the following hold:
\begin{enumerate}[label=(\alph*)]
\item  If $\M^1$ is hyperbolic (diagonalizable with real eigenvalues), then
$\ast$ is uniformly hyperbolic.\label{lem:uniform_hyperbolicity}

\item If $\M^1$ is diagonalizable with eigenvalues
$\phi_1(\lambda_0),\phi_1(\lambda_1),\dots,\phi_1(\lambda_{K})$.
Then, $\P(\hp)$ is diagonalizable with
eigenvalues $p(\lambda_0),p(\lambda_1),\dots p(\lambda_K)$.
\label{lem:eigvals_M1_eigvals_P}

\item  If $\M^1$ is diagonalizable with eigenvalues
$\{\phi_1(\lambda_i)\}_{i=0}^K$ and $\lambda_i\in\Xi$, 
then for every positive $p\in P_K(\Xi)$, all eigenvalues of $\P(\hp)$ are positive.
\label{lem:positivity_of_as_products}
\item If $\zeta$ is a root of $\phi_1\ast \phi_K-\phi_1\phi_K$, 
then $\phi_1(\zeta)$ is an eigenvalue of $\M^1$.
\label{lem:roots_characterization}
\end{enumerate}
\end{lemma}

\begin{proof}
The proof of \ref{lem:uniform_hyperbolicity} and
\ref{lem:eigvals_M1_eigvals_P} follows directly from the recurrence relation
\eqref{eqn:Mk_rec_relation} and \eqref{eqn:three_term_recurrence} since all
matrices $\{\M^k\}_{k=0}^K$ are polynomials in $\M^1$. The proof of
\ref{lem:positivity_of_as_products} follows directly from
\ref{lem:eigvals_M1_eigvals_P} since the eigenvalues of $\P(\hp)$ are given by
$p(\lambda_i)$ for $i=0,1,\dots K$, and $p(\lambda_i)>0$ for all $i$ since $p>0$ on $\Xi$. 
Lastly, let $\zeta$ be a root of $\phi_1\ast
\phi_K-\phi_1\phi_K$, and let $\mathbf{v}$ be the row vector
$(\phi_0(\zeta),\phi_1(\zeta),\cdots,\phi_K(\zeta))$, then by the definition
of $\M^1$,
\[
\mathbf{v}\M^1
=
\begin{bmatrix}
\phi_1(\zeta)\phi_0(\zeta)&
\phi_1(\zeta)\phi_1(\zeta)&
\cdots &
\phi_{1}(\zeta)\phi_{K-1}(\zeta)&
(\phi_1\ast\phi_K)(\zeta)
\end{bmatrix}
=
\phi_1(\zeta)\mathbf{v},
\]
since $\phi_1\ast \phi_K(\zeta)=\phi_1(\zeta)\phi_K(\zeta)$. Hence, $\mathbf{v}$
is a left eigenvector of $\M^1$ with eigenvalue $\phi_1(\zeta)$, which proves
\ref{lem:roots_characterization}.
\end{proof}

An immediate consequence of this lemma is that a truncated product is uniformly
hyperbolic if it is associative and $\phi_1 \phi_K-\phi_1\ast \phi_K$ has $K+1$
distinct roots. This allows us to give a complete characterization of the
eigendecomposition of $\P(\hp)$ as the following theorem shows.
\begin{theorem}
  Let $\ast$ be an ATP on $P_K(\Xi)$, and assume that
$\phi_1\ast\phi_K-\phi_1\phi_K$ has $K+1$ distinct real roots
$\xi_0,\dots,\xi_K$. Define the matrix $V\in\mathbb{R}^{(K+1)\times(K+1)}$ by
$V_{i+1,j+1}:=\phi_j(\xi_i)$.
Then the following statements hold.

\begin{enumerate}[label=(\alph*)]
\item The (left) eigendecomposition of $P(\hp)$ is given by
\[
\P(\hp) = V^{-1} \operatorname{diag}\left(p(\xi_0), p(\xi_1), \dots,
p(\xi_{K})\right) V,\quad \forall \hp\in \mathbb{R}^{K+1}.
\]

\item For all $p,q\in P_K(\Xi)$, we have $(p\ast q)(\xi_i)=p(\xi_i)q(\xi_i)$, 
$i=0,1,\dots,K$.

\item If, in addition, $\xi_0,\dots,\xi_K\in\Xi$,
then $p\ast q$ is the (unique) interpolant of $pq$ at these nodes $\xi_0,\dots,\xi_K$.
\end{enumerate}
\label{thm:associative_hyperbolic_characterization}
\end{theorem}

\begin{proof}
By \cref{lem:Results_about_ATP}\ref{lem:roots_characterization}, each
$\phi_1(\xi_i)$ is an eigenvalue of $\M^1$. Since the roots $\xi_0,\dots,\xi_K$
are distinct, the rows of $V$ are
linearly independent left eigenvectors of $\M^1$, and therefore $V$ is invertible.
Now let $p\in P_K(\Xi)$ with coefficient vector $\hp$. By
\cref{lem:Results_about_ATP}\ref{lem:eigvals_M1_eigvals_P}, the $i$th row of
$V$ is a left eigenvector of $\P(\hp)$ associated with the eigenvalue
$p(\xi_i)$, which proves part (a).

To prove the second claim, we observe that
$V\hp=\left(p(\xi_0),p(\xi_1),\dots,p(\xi_K)\right)^T$. Then
\[
\begin{bmatrix}
(p\ast q)(\xi_0)\\
\vdots \\
(p\ast q)(\xi_K)
\end{bmatrix}
=V\widehat{p\ast q} = V \P(\hp) \hq
= \operatorname{diag}\left(p(\xi_0), \dots, p(\xi_{K})\right)V\hq
=
\begin{bmatrix}
p(\xi_0)q(\xi_0)\\
\vdots \\
p(\xi_K)q(\xi_K)
\end{bmatrix}.
\]
Finally, if the nodes $\xi_0,\dots,\xi_K$ belong to $\Xi$, then $p\ast q$ is a
polynomial in $P_K(\Xi)$ that agrees with $pq$ at $K+1$ distinct points.
Therefore, it is the unique interpolant of $pq$ at these nodes.
\end{proof}

The theorem above shows that if $\phi_1\ast \phi_K -\phi_1\phi_K$ has $K+1$
distinct real roots in $\Xi$, then the ATP is precisely the
interpolant of $pq$ at these roots. In the next subsection, we will discuss
this idea in more detail and interpret such products as collocation products. In
\cref{sec:decouplability_and_collocation}, we will expand on this idea further
and show that under some conditions such as the absence of a source term, the SG
method can be rewritten as a collocation method at the roots of
$\phi_1\ast \phi_K -\phi_1\phi_K$.
Conversely, given a collocation method with $K+1$ distinct
collocation points, we can construct an ATP  
such that the SG method is precisely this collocation method.
Lastly, it is worth noting that the
aforementioned roots may not be real, may
not be distinct, or may not lie in $\Xi$. In such cases, the ATP
cannot be interpreted as an interpolation product on $\Xi$.

\subsection{Examples of associative truncated products}
\label{sec:examples_associative_products}
As discussed before, every choice of $\phi_1\ast \phi_K$ corresponds to a unique
ATP. In this section, we will discuss a few ATPs and their properties:
\begin{enumerate}[label=(\alph*)]
\item The {\it zero product} $\pzer$  defined by choosing $\phi_1\ast \phi_K=0$.
This product is not symmetric since $\M^1$ is not symmetric. 
Moreover, by \cref{lem:Results_about_ATP}\ref{lem:roots_characterization}, 
the eigenvalues of $\M^1$ are of the form
$\phi_1(\zeta)$, where $\zeta$ is a root of $\phi_1 \phi_K$.
However, this does not imply that it is hyperbolic since $\phi_1\phi_K$ may
have double roots, as
is the case with Legendre polynomials for odd values of $K$,  and it is possible
that some of these eigenvalues are defective.
\item The {\it $L^\infty$ product} $\pinf$ defined by choosing $\phi_1\ast
\phi_K$ to be the closest polynomial
in $P_K(\Xi)$ to $\phi_1\phi_K$ in the $\norm{\cdot}_{\infty}$ assuming that
$\Xi$ is compact. This polynomial can be
obtained numerically using the Remez algorithm
\cite{powellApproximationTheoryMethods2001}. By the
equioscillation theorem, $\phi_1\pinf\phi_K-\phi_1\phi_K$ has $K+1$ distinct
roots in $\Xi$, which implies that it is hyperbolic by
\cref{lem:Results_about_ATP}\ref{lem:roots_characterization}.
Furthermore, if $p>0$ on $\Xi$, then
$\P(\hp)$ will have positive eigenvalues by
\cref{lem:Results_about_ATP}\ref{lem:positivity_of_as_products}. In general, since
$\M^1$ is not symmetric, the $L^\infty$ product is not symmetric.

\item The {\it collocation product} $\co$ defined by choosing $\phi_1\ast
\phi_K$ to be the interpolant of $\phi_1\phi_K$ at a given set of distinct
points $\{\xi_\ell\}_{\ell=0}^K$. Using induction and the three term
recurrence relation, we can show that $(\phi_i\co
\phi_j)(\xi_\ell)=\phi_i(\xi_\ell)\phi_j(\xi_\ell)$ for all
$0\le i,j,\ell\le K$. Furthermore, the eigenvalues of $\M^1$ are precisely
$\{\phi_1(\xi_\ell)\}_{\ell=0}^K$, which implies that the
collocation product is
hyperbolic. In general, the collocation product is not symmetric.

\item  The {\it associative and symmetric (AS) product} $\as$ defined by choosing
$\phi_1\ast \phi_K$ to be the interpolant of $\phi_1\phi_K$ at the $K+1$ roots of
$\phi_{K+1}$, which are the nodes of the Gaussian quadrature rule associated
with $w$ and $\Xi$. By
\cref{thm:associative_hyperbolic_characterization}, it is associative and
hyperbolic.  
We next show that it is also symmetric.
\end{enumerate}

\begin{lemma}\label{lem:AS_product_is_ps_M1}
Let $\as$ be the truncated product defined by interpolation at the roots of
$\phi_{K+1}$. Then
\begin{equation}
\phi_1\as \phi_K=\phi_1\ps \phi_K.
\label{eqn:AS_product_is_ps_M1}
\end{equation}
Hence, the AS product is associative and symmetric.
\end{lemma}

\begin{proof}
Let $\{((\xi_\ell,w_\ell)\}_{\ell=0}^K$ be the nodes and weights of the
Gaussian quadrature rule associated with $w$ and $\Xi$, and let $0\le i\le K$.
Since this Gaussian quadrature rule is exact on polynomials of degree up to
$2K+1$, we have
\begin{align*}
\langle \phi_1\as \phi_K,\phi_i \rangle_w &=
\sum_{\ell=0}^K w_\ell  (\phi_1\as \phi_K)(\xi_\ell)\phi_i(\xi_\ell) =
\sum_{\ell=0}^K w_\ell  \phi_1(\xi_\ell) \phi_K(\xi_\ell)\phi_i(\xi_\ell)\\ &=
\langle \phi_1\phi_K,\phi_i \rangle_w=
\langle \phi_1\ps \phi_K,\phi_i \rangle_w,
\end{align*}
which prove \eqref{eqn:AS_product_is_ps_M1} since the equality above
holds for all $0\le i\le K$.
Therefore, the AS product share the same $\M^1$ as the pseudospectral
product, which is symmetric. Hence, the AS product is associative (by
construction) and
symmetric by \cref{lem:symmetry_of_M_k}.
\end{proof}

\begin{remark}
\label{rem:AS_vs_PS}
Lemma \ref{lem:AS_product_is_ps_M1} shows that the products $\as$ and $\ps$
have the same matrix $\M^1$. In general, however, they do not coincide as
truncated products. Indeed, for the associative and symmetric product,
the higher-order matrices $\M^k$, $k\ge2$, are generated from $\M^1$ by the
recurrence \eqref{eqn:Mk_rec_relation}, which enforces associativity. By
contrast, for the pseudospectral product, the matrices $\M^k$ are defined
directly by projection of $\phi_k\phi_j$ onto $P_K(\Xi)$, and these matrices do
not in general satisfy the commutativity relations required for associativity.
Thus, $\as$ may be viewed as the ATP obtained by
retaining the same first-order multiplication matrix $\M^1$ as the
pseudospectral product while redefining the higher-order matrices so as to
restore associativity.
\end{remark}

We summarize the properties of the truncated products discussed
above in \cref{table:associative_products}. Here, the symbol `?'
indicates that
the property may or may not hold, depending on the distribution $w$ and the
degree of the polynomials $K$.
\begin{table}[htbp]
\centering
\begin{tabular}{|c|c|c|c|c|}
\hline
Product & Associative & Hyperbolic& Symmetric  & Pos. definite \\ \hline
Pseudospectral product &  \xmark&\cmark& \cmark & \cmark \\
\hline
Zero product  & \cmark & ? & \xmark & ? \\
\hline
$L^\infty$ product  &\cmark & \cmark & \xmark & \cmark \\
\hline
Collocation product  & \cmark & \cmark & ? &\cmark\\
\hline
AS product  &\cmark & \cmark & \cmark & \cmark \\
\hline
\end{tabular}
\caption{Some examples of truncated products on $P_K(\Xi)$ with
$\Xi\subset \mathbb{R}$.}
\label{table:associative_products}
\end{table}

\subsection{Convergence of truncated products to the usual product}
To quantify the approximation error introduced by a truncated product, we
consider a sequence of truncated products
$\ast_K : P_K(\Xi)\times P_K(\Xi)\to P_K(\Xi)$.
For each $K\ge 0$, we define $C_K^\ast$ as the smallest constant such that
\begin{equation}\norm{p\ast_K q}_{w,\Xi}\le C_K^\ast \norm{p}_{w,\Xi}
\norm{q}_{w,\Xi},\quad \forall p,q\in P_K(\Xi).
\label{eqn:truncated_product_bound}
\end{equation}
In other words, $C_K^\ast$ is the operator norm of the bilinear map
$(p,q)\mapsto p\ast_K q $.
Note that $C_K^\ast$ is finite since $P_K(\Xi)$ is finite dimensional. 
Now, we are ready to
state the following theorem about the convergence of
truncated products.
\begin{theorem}
\label{thm:truncated_product_accuracy}
Consider $f,g\in L^2_w(\Xi)$ and let $f_K,g_K\in P_K(\Xi)$ be the $L^2_w$ projections 
of $f$ and $g$ onto $P_K(\Xi)$, respectively. Assume that 
\begin{equation}
\lim_{K\to \infty}\norm{fg - f_Kg_K}_{w,\Xi}=0, \qquad 
\lim_{K\to \infty} C_{K}^{\ast}
\Bigl(\|f_K-f_{r}\|_{w,\Xi} + \|g_K-g_{r}\|_{w,\Xi} \Bigr) = 0,
\label{eqn:truncated_product_accuracy_CK_condition}
\end{equation}
where $r:=\left\lfloor K/2\right\rfloor$.
Then 
\[\lim_{K\to \infty}\norm{fg - f_K\ast_K g_K}_{w,\Xi}=0.\]
\end{theorem}
\begin{proof} 
Since $r=\lfloor K/2\rfloor$, we have $\deg(f_{r}g_{r})\le 2r\le K$, and $f_r\ast_K g_r= f_rg_{r}$ following the consistency property of $\ast_K$. 
Using this identity, we write 
\begin{equation}
\norm{fg-f_K\ast_K g_K}_{w,\Xi}\le
\norm{fg-f_{r}g_{r}}_{w,\Xi}+
\norm{f_{r}\ast_K g_{r}-f_K\ast_K g_K}_{w,\Xi}.
\label{eqn:truncated_product_accuracy_split}
\end{equation}
It remains to estimate the second term. By bilinearity and
\eqref{eqn:truncated_product_bound},
\begin{align}
\|f_K\ast_K g_K-f_{r}\ast_K g_{r}\|_{w,\Xi}
&\le
C_K^\ast\Bigl( \|f_K-f_{r}\|_{w,\Xi}\|g_{r}\|_{w,\Xi} + \|f_K\|_{w,\Xi}\|g_K-g_{r}\|_{w,\Xi} \Bigr) \notag \\
&\le
C_K^\ast\Bigl( \|g\|_{w,\Xi}\|f_K-f_{r}\|_{w,\Xi} + \|f\|_{w,\Xi}\|g_K-g_{r}\|_{w,\Xi} \Bigr),
\end{align}
where the last inequality follows from the fact that $\Pi_K$ is the orthogonal projection in $L^2_w(\Xi)$ which leads to $\|f_K\|_{w,\Xi}\le \|f\|_{w,\Xi}$,  $\|g_{r}\|_{w,\Xi}\le \|g\|_{w,\Xi}$.
The conclusion now follows from \eqref{eqn:truncated_product_accuracy_split}.
\end{proof}

To illustrate the preceding result, we consider the following numerical
experiment. Let $\Xi=[-1,1]$
with $w\equiv 1$, and let $m,\rho,f: \Xi \to \R$ be defined as
\[
m(\xi)=\sin(\pi \xi),\quad \rho(\xi)=1+0.2\cos(\xi),\quad
f(\xi)=\frac{m(\xi)^2}{\rho(\xi)}.
\]
Let $m_K=\Pi_K m$ and $\rho_K=\Pi_K\rho$, then we define $f_K$ as
$\G_K\left(\P(\widehat{\rho_K})^{-1}\P(\widehat{m_K})\widehat{m_K}\right)$,
where $\ast$ is a truncated product. Then, we calculate the norm of the
residual. The results are presented in
\cref{fig:truncated_product_accuracy} for several choices 
of truncated products and polynomial degree $K$. From the results, we
observed that the pseudospectral product, the AS product, the
$L^\infty$ product, and
the zero product all have comparable approximation errors, and the polynomial $f_K$
obtained from each of these products converges to $f$ as $K$ increases.

\begin{figure}[htbp]
\includegraphics[scale=1]{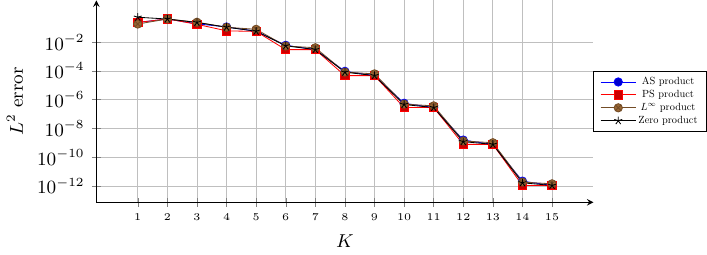}
\caption{The $L^2$ norm of the residual
$\norm{f-f_K}_{w,\Xi}$ for different truncated products on
$P_K(\Xi)$ with $\Xi=[-1,1]$ and $w\equiv 1$.
Here, PS denotes ``pseudospectral''.
}
\label{fig:truncated_product_accuracy}
\end{figure}

\subsection{Extension to multidimensional random variables}
\label{sec:multi_dimensional_truncated_products}
The results in the previous subsection extend naturally to the case of
multidimensional random variables via tensor products. More precisely, let
$\Xi=\Xi_1\times \Xi_2 \times \cdots \times \Xi_d\subset \R^d$, and let
$w(\xi)=\prod_{i=1}^d w_i(\xi_i)$ be a weight function on $\Xi$. Then, we define
$\phi_{\mathbf{k}}(\xi)=\prod_{i=1}^d \phi_{k_i}(\xi_i)$, where
$\mathbf{k}=(k_1,k_2,\dots,k_d)$ and $\phi_{k_i}$ is the $k_i$-th orthonormal
polynomial with respect to $w_i$ on $\Xi_i$. Now, given truncated products
$\ast_1,\ast_2,\dots,\ast_d$ on $P_{K}(\Xi_1),P_{K}(\Xi_2),\dots,
P_{K}(\Xi_d)$, respectively, we define the truncated product $\ast$ on the tensor product space
$Q_{K}(\Xi)$ as follows: For any two basis functions
$\phi_{\mathbf{i}},\phi_{\mathbf{j}}$ with
$\mathbf{i}=(i_1,\dots,i_d)$ and
$\mathbf{j}=(j_1,\dots,j_d)$, we define
\begin{equation}\phi_{\mathbf{i}}\ast \phi_{\mathbf{j}}=
\prod_{\ell=1}^d
\left(\phi_{i_\ell}\ast_\ell \phi_{j_\ell}\right).
\label{eqn:multi_dimensional_truncated_product_def}
\end{equation}
Then, by bilinearity, we can extend $\ast$ to all elements of $Q_K(\Xi)$.
It follows directly from the definition that $\ast$ is associative if each
underlying products $\ast_\ell$ is associative. Similarly, the
product $\ast$ inherit the symmetry and the uniform hyperbolicity
properties of the
underlying products.

The associated
matrices $\{\M^\mathbf{k}\}$ for $\mathbf{k}\in \{0,1,\dots,K\}^d$, can be
obtained directly from the matrices associated with the products
$\{\ast_{\ell}\}_{\ell=1}^d$ via a Kronecker product.
More specifically, if $\{\M^{(\ell),k}\}_{k=0}^K$ denotes the one-dimensional matrices associated
with the product $\ast_{\ell}$, then we can calculate $\M^{\mathbf{k}}$ directly
using the following formula
\begin{equation}
\M^{\mathbf{k}} = \M^{(1),k_1}\otimes
\M^{(2),k_2}\otimes \cdots \M^{(d),k_d},
\quad \mathbf{k}=(k_1,k_2,\dots,k_d)\in \{0,1,\dots,K\}^d.
\label{eqn:multi_dimensional_truncated_product_Mk}
\end{equation}
where $\otimes$ denotes the Kronecker product of matrices. 

This construction provides a direct way to obtain ATPs
on the tensor-product space $Q_K(\Xi)$. However,
it remains unclear whether a simple construction of ATPs on the total-degree space 
$P_K(\Xi)$ exists and, if it does, whether the resulting product
retains the properties described in
the previous subsection (e.g., the explicit description of the eigenvalues of
$\P(\hp)$ and the positive-definiteness criteria).
This difficulty was encountered before in the context of the pseudospectral product
when extending the Haar wavelets (on which the pseudospectral product is
associative) to multiple dimensions using a sparse basis (see Section 3.1.3 in
\cite{gersterStabilizationUncertaintyQuantification2020}). { The
constructed of
ATPs on sparse or total-degree bases in multiple dimensions is definitely
interesting and useful, but it is beyond the scope of the present work.}

\subsection{Summary of the construction of associative truncated products}
\label{subsec:summary_of_construction_procedure}
To summarize, we construct an ATP on $\PK$ by establishing the
matrices $\{\M^k\}_{k}$ associated with the product, then the truncated product
of two functions is calculated in terms of the coefficients of the functions and
the matrices $\{\M^k\}_{k}$ as described in \cref{eqn:pq_M_form}. 
The construction can be reduced to the following steps:
\begin{enumerate}
\item Set $\M^0$ to be the identity matrix.
\item In the case of one random variable: Choose the element $\phi_1\ast\phi_K\in P_K(\Xi)$ and construct the matrix $\M^1$ with $\M^1_{i,j}=\langle
\phi_1\ast\phi_i,\phi_j\rangle_w$.
In the case of multiple random variables with a tensor product
structure: Construct $\M^{(1,0,\dots,0)},\M^{(0,1,0,\dots,0)},\dots,
\M^{(0,\dots,0,1)}$ using
\eqref{eqn:multi_dimensional_truncated_product_Mk}.
\item In the case of one random variable: Use the three term
recurrence relation \cref{eqn:Mk_rec_relation}
to compute the rest of the matrices $\{\M^k\}_{k=2}^K$. In the case of
multiple random variables with a tensor product structure: Use
\eqref{eqn:multi_dimensional_truncated_product_Mk}.
\end{enumerate}

Once the matrices $\{\M^k\}_{k=0}^K$ have been assembled, we store them for use in
\eqref{eqn:pq_M_form} to calculate the truncated product of any two functions in
$\PK$. It is worth noting here that the second step is the most
crucial step in the construction procedure since it involves the choice of the
truncated product. In comparison, the third step is a sequence of matrix
multiplications.

\section{Hyperbolic systems of conservation laws with uncertainty}
\label{sec:hyperbolic_systems}

The main motivation for introducing the notion of associative truncated products
is to construct hyperbolic stochastic Galerkin systems for nonlinear hyperbolic systems of
conservation laws with uncertainty. In this section, we will discuss
a general hyperbolicity criterion for SG systems built from ATPs, and then 
apply it to construct hyperbolic SG systems
for the isothermal Euler equations, and the Euler equations with uncertain
initial conditions.

Before describing the main result, we introduce the nonlinear operations induced by
a truncated product. Let $\ast$ be a truncated product on $\PK$, and let
$a,b\in \PK$. Whenever $\P(\hat b)$ is invertible, we define the truncated
division of $a$ by $b$ by $a\div b:=\G_K\bigl(\P(\hat b)^{-1}\hat a\bigr)$. 
Likewise, whenever there exists $c\in \PK$ such that $a=c\ast c$,
we write $c=\sqrt[\ast]{a}$ and call $c$ a truncated square root of $a$.
In the applications considered below, the positivity assumptions will guarantee
that the relevant matrix inverses and matrix square roots are well-defined.

To simplify the presentation, we will use $\mathcal{A}$ to
denote the class of functions that can be represented by a composition of
rational functions and square roots.
For any $f:\mathbb{R}^n \to \mathbb{R}^n$ with entries belonging to
$\mathcal{A}$, we define $f^*:(\PK)^n \to
(\PK)^n$ by replacing the usual product, division, and square root in each component of
$f$ by the truncated operations $\ast$, $\div$, and $\sqrt[\ast]{\,\cdot\,}$,
respectively. Additionally, we
use $\hyp(f)$ to denote the set of points
$u\in \mathbb{R}^n$ where the Jacobian of $f$ is hyperbolic.

\begin{theorem}\label{thm:rational_function_approximation}
Let $f=(f_i)_{i=1}^n:\mathbb{R}^n \to \mathbb{R}^n$ be a flux whose entries are rational functions. 
Assume that its hyperbolicity region is of the form 
\[
\hyp(f)=\{u\in \mathbb{R}^n: p_i(u) > 0, \, i=1,\dots,m\},
\]
where each $p_i$ is a polynomial. Furthermore, assume that each entry of $S(u)$, the matrix of
eigenvectors of the Jacobian of $f$ at $u\in \hyp(f)$, belongs to $\mathcal{A}$. 

Let $\ast$ be an associative, hyperbolic and positive definite
truncated product on $\PK$, and let
$\what{f^{\ast}}:\mathbb{R}^{n(K+1)}\to
\mathbb{R}^{n(K+1)}$, be
the coefficients of $f^*$ (defined above) in the basis $\{\phi_k\}_{k=0}^K$ and
\[
\mathcal{Z} :=
\{(\what{u_1},\dots,\what{u_n})\in \mathbb{R}^{n(K+1)}:
p^\ast_i(u_1,\dots,u_n)(\xi) > 0, i=1,\ \dots,m,
\forall \xi \in \Xi\},
\] 
then
\[
\mathcal Z\subset \hyp(\widehat{f^\ast}),
\]
i.e., the Jacobian of the stochastic Galerkin flux is hyperbolic at every state in $\mathcal Z$.
\end{theorem}

\begin{proof}
We first record how differentiation behaves under the truncated operations.
Since $\ast$ is associative, $\P(\widehat{p\ast q})=\P(\hat p)\P(\hat q)$.
Also, commutativity of $\ast$ implies that the matrices
$\P(\hat p)$ and $\P(\hat q)$ commute. Thus products of elements in $\PK$ may be
represented equivalently by products of their multiplication matrices.
Consider a monomial of the form $\psi(u_1,u_2)=u_1^{a}u_2$ with $a\ge
1$, then
$\widehat{\psi^\ast}(\what{u_1},\what{u_2})=\P(\what{u_1})^{a}\what{u_2}$. 
Differentiating this expression with respect to $\what{u_1}$ and using the
commutativity of the multiplication matrices gives
$$
\partial_{\what{u_1}} \widehat{\psi^\ast}\left(\what{u_1},\what{u_2}\right)
=a\P(\what{u_1})^{a-1} \P(\what{u_2})
=\partial_{u_1}\psi(\P(\what{u_1}),\P(\what{u_2})).
$$
The same argument applies to each monomial, and by linearity we obtain, for any
polynomial $\psi:\mathbb{R}^n\to\mathbb{R}$,
\[
\partial_{\what u_j}\widehat{\psi^\ast}(\what u_1,\dots,\what u_n)
=
\partial_{u_j}\psi(\P(\what u_1),\dots,\P(\what u_n)).
\]
Applying the same argument together with the quotient rule gives, for each
rational component $f_i$,
\begin{equation}
\partial_{\what u_j}\widehat{f_i^\ast}(\what{u_1},\dots,\what{u_n})
=
\partial_{u_j}f_i\bigl(\P(\what{u_1}),\dots,\P(\what{u_n})\bigr).
\label{eqn:lifted_jacobian_entries}
\end{equation}

By applying this idea to all components of $f$, we 
can express the Jacobian of $\what{f^*}$ at the point
$(\what{u_1},\dots,\what{u_n})$, denoted $J_{\what{f^*}}$, as follows
\begin{equation}
J_{\what{f^*}}=
\begin{bmatrix}
  \partial_{\what{u_1}}\what{f_1^*} & \cdots & \partial_{\what{u_n}}\what{f_1^*}\\
\vdots & \ddots & \vdots\\
\partial_{\what{u_1}}\what{f_n^*} & \cdots & \partial_{\what{u_n}}\what{f_n^*}
\end{bmatrix}
=J_f\left(\P(\what{u_1}),\dots, \P(\what{u_n})\right).
\label{eqn:jacobian_of_f_star_block_structure}
\end{equation}
Now, consider $\hat{u}=(\what{u_1},\dots,\what{u_n})\in \mathcal{Z}$, then by
the associativity and positivity of $\ast$, all eigenvalues of
$p_i(\P(\what{u_1}),\dots,\P(\what{u_n}))$  are positive. Also, since each entry of $S$ is in $\mathcal{A}$ and $S$ is invertible,
each entry
of  $S^{-1}$ is in $\mathcal{A}$ as well.
Hence, the inverse of  $W=S(\P(\widehat{u_1}),\dots,\P(\widehat{u_n}))$  is given by 
$W^{-1}=
S^{-1}(\P(\widehat{u_1}),\dots,\P(\widehat{u_n}))
$.

Since $\ast$ is associative and commutative, it follows then that the set of matrices $\{\P(\x)\}_{\x\in \mathbb{R}^{K+1}}$
forms a commutative algebra with the usual matrix addition and multiplication. Therefore, 
 the matrix  $S^{-1}J_fS=\Lambda(u_1,\cdots,u_n)$ being diagonal implies
$W^{-1}J_{\what{f^*}}W=$ $\Lambda\bigl(\P(\what u_1),\dots,\P(\what u_n)\bigr)$ is bloc diagonal and that each diagonal entry is of the form
$g(\P(\hu_1),\dots,\P(\hu_n))$ where $g \in \mathcal{A}$.

Finally, since $\ast$ is
uniformly hyperbolic (since it is associative and hyperbolic), there is a
matrix $V$ such as $V^{-1}\P(\hu_i)V$ is diagonal for all $i=1,\dots,n$.
Therefore, $W\diag(V,\dots,V)$ diagonalizes $J_{\what{f^*}}$, which
implies that $J_{\what{f^*}}$ is hyperbolic.
\end{proof}

Examples of system that satisfy the assumptions of \cref{thm:rational_function_approximation} include the
isothermal Euler equations, the Euler equations, and the shallow water
equations. We will discuss the first two systems in more detail in the next two
subsections. 

\subsection{The isothermal Euler equations}
\label{sec:isothermal_euler}
Consider the isothermal Euler equation with friction source term on a domain $(a,b)\times (0,T)$,
\begin{equation}
\p_t(\rho,m) +\p_x(m,\rho +m^2\rho^{-1}) =(0,\fr(\rho,m))
\label{eqn:isothermal_euler}
\end{equation}
where $\rho,m,\fr$ represent the density, momentum and friction, respectively.
We assume that $\fr$ is modeled by the steady state
friction
\cite{bandaGasFlowPipeline2006,gugatIsothermalEulerEquations2017,
zhouSimulationTransientsNatural2000}
\[
\fr(\rho,m)= -f_g \frac{m|m|}{2D\rho},
\]
where $f_g$ is the steady state friction factor and $D$ is the diameter of the
pipeline. 

The stochastic version of \eqref{eqn:isothermal_euler} accounts for
uncertainties in $\rho$, $m$ or $\fr$. That is, $\rho$ and $m$ are now
functions of $x,t$ and the random variable $\xi$, and they are
subject to a given (noisy) initial condition $\rho(x,0,\xi)=\rho_0(x,\xi)$  and
$m(x,0,\xi)=m_0(x,\xi)$.  In the SG setting, we approximate $\rho$ and $m$ 
by $\rho_K,m_K\in \PK$, with coefficient
vectors $\hrho$ and $\hm$. Using the truncated product $\ast$ for
the nonlinear flux term gives the stochastic Galerkin system
\begin{equation}
\p_t (\hrho,\hm) + \p_x(\hm,\hrho+
\P(\hm)\P(\hrho)^{-1}\hm)=\left(\zero,\widehat{\fr}(\hrho,\hm)\right),
\label{eqn:isothermal_vector_form}
\end{equation}
where $\widehat{\fr}(\hrho,\hm)$ is the vector of coefficients of
the projection of
$\fr(\rho_K,m_K)$ onto  $\PK$. Thus, the stochastic isothermal Euler equation is
reduced to a system of $2(K+1)$ deterministic PDEs. By direct calculation, we
know that the Jacobian of the flux $f(\rho,m)=(m,\rho+m^2\rho^{-1})$ is
hyperbolic for $\rho >0$. Hence, by
\cref{thm:rational_function_approximation}, the Jacobian of flux function in the
SG system $\hat{f}(\hrho,\hm)=(\hm,\hrho+\P(\hm)\P(\hrho)^{-1}\hm)$ is
hyperbolic for all $\hat{\rho}$ such that $\G_K(\hrho)>0$ as stated
in the following lemma.

\begin{lemma}
Let $\ast$ be an associative, hyperbolic and positive definite truncated product
on $\PK$, then the SG system \eqref{eqn:isothermal_vector_form} is
hyperbolic for all $\hat{\rho}$ such that $\G_K(\hrho)>0$.
\end{lemma}

The positivity condition on $\G_K(\hrho)$ ensures, by positive definiteness of
the truncated product, that all eigenvalues of $\P(\hrho)$ are positive; in
particular, $\P(\hrho)$ is invertible and the flux in
\eqref{eqn:isothermal_vector_form} is well defined. 
In particular, the assumption
that $\ast$ is associative is crucial, since it is known that the pseudospectral
product, which is hyperbolic and positive definite, could lead to a
non-hyperbolic
system when $K\ge 2$ even if the pressure is positive 
\cite{jinStudyHyperbolicityKinetic2019}.
This does not contradict the lemma above since the pseudospectral product is
only associative when $K\le 1$.

\subsection{Collocation Interpretation and Decoupling}
\label{sec:decouplability_and_collocation}
The SG system \eqref{eqn:isothermal_vector_form} cannot, in general, be
decoupled into $2(K+1)$ independent systems of PDEs due to the coupling among the stochastic modes 
introduced by the friction term. However, if the friction term is zero, then the
system admits a collocation interpretation for
the class of products described in
\cref{thm:associative_hyperbolic_characterization}
and may be decoupled. Assume that $\Xi\subset \mathbb{R}$ and $\ast$
is associative and hyperbolic truncated product on $\PK$ such that $\phi_1\ast
\phi_K-\phi_1\phi_K$ has $K+1$ distinct roots $\{\xi_i\}_{i=0}^K$,
then it follows from
\cref{thm:associative_hyperbolic_characterization} that $\ast$ corresponds to a
collocation product with collocation points $\{\xi_i\}_{i=0}^K$. Now, let
$\check{\rho}_i$ and $\check{m}_i$ be the values of $\G_K(\hrho)$ and
$\G_K(\hm)$ at the
collocation point $\xi_i$, then by applying $\cdot \mapsto
\G_K(\cdot)(\xi_i)$ to the system \eqref{eqn:isothermal_vector_form} with
$f_g=0$,
we obtain the following system of PDEs for $\check{\rho}_i$ and $\check{m}_i$:
\begin{equation}
\p_t(\check{\rho}_i,\check{m}_i) +\p_x\left(\check{m}_i,\check{\rho}_i
+\frac{\check{m}_i^2}{\check{\rho}_i}\right)=0.
\end{equation}
Thus, the stochastic Galerkin system written in the variables
$\{(\check{\rho}_i,\check{m}_i)\}_{i=0}^K$ is equivalent to applying the
stochastic collocation (SC) method at the 
collocation points $\{\xi_i\}_{i=0}^K$. 
Conversely, given any set of $K+1$ distinct collocation points $\{\xi_i\}_{i=0}^K$, we can
construct a collocation product as discussed in
\cref{sec:examples_associative_products}. This product is associative and
hyperbolic, and positive definite. Therefore, for the
homogeneous isothermal Euler equations, the SG system generated by this product
is hyperbolic and it is equivalent, after the change of variables from coefficient values to nodal
values, to the SC system at the same nodes.
This is illustrated by the following  diagram:

\begin{center}
\begin{tikzcd}[column sep=20ex]
\text{Collocation points } \{\xi_i\}
\arrow[r, shift left,
"(p\co q)(\xi_i)=p(\xi_i)q(\xi_i)"]
\arrow[d, ] &
\text{Collocation product $\co$}
\arrow[d ] \arrow[l,shift left,
"\text{roots of }\phi_1\ast\phi_K-\phi_1\phi_K"]\\
\text{SC in nodal variables }(\check{\rho},\check{m})
\arrow[r,<->,dashed ]
& \text{SG in coefficient variables }(\hrho,\hm)
\end{tikzcd}
\end{center}

The equivalence above should be understood at the level of the homogeneous
semi-discrete stochastic Galerkin system after a change of variables from modal
coefficients to nodal values. Thus, for (only) a collocation product, the SG formulation
recovers the same decoupled PDEs as stochastic collocation at the corresponding
nodes. The viewpoint is nevertheless different from applying SC directly. In the
present approach, the collocation structure arises from the choice of an
ATP within an intrusive SG formulation, and the same
algebraic framework can also be used when the system contains source terms or
other operations that couple the stochastic modes. At the fully discrete level,
the two implementations may also differ: a direct SC method often advances each
sample independently, possibly with sample-dependent time steps or solver
choices, whereas the SG formulation is advanced as a single coefficient system,
typically with one global time step determined by the largest wave speed over
the stochastic modes.

\subsection{Euler equations in one space dimension}
\label{sec:euler_1d}
In this section, we consider the one-dimensional Euler equations
for an ideal gas, which can be written in conservative form as
\begin{equation}
(\rho,m,E)_t + \left(m,\frac{m^2}{\rho}+p,\frac{m}{\rho}(E+p)\right)_x = 0,
\label{eqn:euler_deterministic}
\end{equation}
where $\rho$ is the density, $m=\rho u$ is the momentum, $u$ is the velocity,
$E$ is the total energy, and $p$ is the pressure. The system is closed
by the equation of state for an ideal gas:
\begin{equation}
p = (\gamma - 1)\left(E - \frac{m^2}{2\rho}\right),
\label{eqn:ideal_gas_law}
\end{equation}
where $\gamma > 1$ is the adiabatic coefficient. 
The deterministic system is hyperbolic on the admissible set
$\rho>0$ and $p>0$,
or equivalently $\rho>0$ and $2\rho E-m^2>0$.
Following the same ideas as
before, we assume that $\u=(\rho, m, E)$ depend on a random variable $\xi$ with
a known distribution, and we seek the SG solution $\hat{\u}(x,t)=(\hrho,\hm,\hE)$ in the form of a polynomial chaos expansion with coefficients. 
Using the truncated product $\ast$, the SG system for Euler equations can be written as 
\begin{equation}
\p_t(\hrho,\hm,\hE) + \p_x\left(\hm,\P(\hm)\P(\hrho)^{-1}\hm+\hp,
\P(\hm)\P(\hrho)^{-1}\left(\hE+\hp\and\right)\right) = 0,
\label{eqn:Euler_SG_formulation}
\end{equation}
where $\hp=(\gamma-1)(\hE-\tfrac{1}{2}\P(\hm)\P(\hrho)^{-1}\hm)$. 
\begin{lemma}
\label{lem:hyperbolicity_of_euler_SG}
Let $\ast$ be an associative, hyperbolic and positive definite truncated product
on $\PK$, then the SG system \eqref{eqn:Euler_SG_formulation} is
hyperbolic for all $(\hrho,\hm,\hE)$ satisfying $\G_K(\hrho)>0$ and
$\G_K(\hp)>0$.
\end{lemma}
Although this result can be derived from
\cref{thm:rational_function_approximation} directly, we will provide a proof 
below to illustrate the construction of the matrix $W$ 
in the proof of   the aforementioned theorem.

\begin{proof}
The argument here follows the same lines as argument for the
hyperbolicity of the
deterministic system \eqref{eqn:euler_deterministic}.
Let $\what{f^*}$ be the flux in \eqref{eqn:Euler_SG_formulation},  
$J_{\what{f^*}}$ be its Jacobian.  We define
$\sU=\P(\hrho)^{-1}\P(\hm)$, $\sP=\P(\hp)$, $\sR=\P(\hrho)$, $\sM=\P(\hm)$, $\sE=\P(\hE)$
and $\sH=(\sE+\sP)\sR^{-1}$. Then,
by a direct computation, we have
\[J_{\what{f^*}}=
\frac{\partial {\what{f^*}}}{\partial \huu} =
\begin{bmatrix}
O & I & O  \\
\displaystyle \frac{\gamma - 3}{2} \sU^2 & (3 - \gamma)\sU & (\gamma - 1)I \\
\displaystyle \frac{\gamma - 1}{2}\sU^3  -
\sU\sH
&
\displaystyle  \sH - (\gamma - 1) \sU^2 & \gamma \sU
\end{bmatrix},\]
where $O$ is the zero matrix.
Now, let $\sC$ be a square root of ${\gamma \sP\sR^{-1}}$, then
\[
W^{-1}J_{\what{f^*}} W=\diag(\sU-\sC,\sU,\sU+\sC),\quad
\text{where }\ W=
\begin{bmatrix}
I & I & I \\
\sU-\sC & \sU & \sU+\sC \\
\sH-\sU\sC & \frac{1}{2}\sU^2 & \sH+\sU\sC
\end{bmatrix}.
\]
It follows from the associativity of $\ast$ that $\sU$ and $\sC$ are simultaneously diagonalizable by a matrix $V$. Thus, the matrix $W\diag(V,V,V)$ diagonalizes the matrix $J_{\what{f^*}}$.
\end{proof}
We note that in the proof, we only required $\sC$ to have a square root. This
square root need not be unique or symmetric positive definite for the lemma
to hold. For instance, the lemma holds for the $L^\infty$ truncated product
$\pinf$ provided that the conditions on $\P(\hrho)$ and $\P(\hp)$ are met.

\subsection{Time integration}
\label{sec:time_integration}
The systems \eqref{eqn:isothermal_vector_form} and
\eqref{eqn:Euler_SG_formulation} are systems of hyperbolic PDEs of the form
$\huu_t + \hff(\huu)_x = \hat{S}(\huu)$, where $\huu$ is the vector of SG coefficients,
$\hff$ is the SG flux function, and $\hat{S}$ is the projected source term. In the numerical experiments below, we use a first-order finite volume spatial discretization 
with the local Lax-Friedrichs flux and a forward Euler time stepping. More precisely,
with a given cell
$I_i=[x_{i-1/2},x_{i+1/2}]$ and time $t_n$, we will update the
solution using 
\begin{subequations}
\label{eqns:time_discretization}
\begin{equation}
\label{eqn:fully_discrete_FV}
\frac{\huu_j^{n+1}-\huu_j^n}{\Delta t}
+\frac{\widehat{\mathcal{F}}_{j+\frac{1}{2}}^n
-\widehat{\mathcal{F}}_{j-\frac{1}{2}}^n}{\Delta x} = \widehat{S}_j^n,
\end{equation}
where $\widehat{\mathcal{F}}_{j+1/2}^n$ is the numerical flux taking the form
\begin{equation}
\label{eqn:lax_friedrichs_flux}
\widehat{\mathcal{F}}_{j+\frac{1}{2}}^n =
\frac{1}{2}\left(\hff(\huu_j^n)+\hff(\huu_{j+1}^n)\right)
-\frac{a_{j+\frac12}^n}{2}\left(\huu_{j+1}^n-\huu_j^n\right).
\end{equation}
\end{subequations}
Here, $a_{j+1/2}^n$ is the maximum of the spectral radius of the Jacobian of
$\hff$ at time $t_n$ over the cells $I_j$ and $I_{j+1}$. The source
term $\widehat{S}_j^n$ is evaluated at
the center of the cell $I_j$ at time $t_n$. The time step size $\Delta t$ is chosen to satisfy
the CFL condition, i.e., $\Delta t \le \mathtt{CFL}\frac{\Delta x}{\max_j a_{j+1/2}^n}$,
with $\mathtt{CFL}<1$.

\section{Numerical results}
\label{sec:numerical_results}
In this section, we present some numerical experiments highlighting the
accuracy and the hyperbolicity-preserving property of the SG systems constructed
using an ATP. In all examples,  we observed that  
computed stochastic Galerkin states remained in the admissible set identified in the preceding
analysis, and the corresponding flux Jacobians remained hyperbolic throughout
the computation.

In the following numerical experiments,
the physical domain is discretized into $N$ uniform cells, and the uncertainty is
represented by a polynomial chaos expansion of degree $K$ with coefficients
vector $\hat{\u}$. We use the first order local Lax-Friedrichs and forward Euler method 
described in \cref{sec:time_integration} with a CFL constant
$\mathtt{CFL}=0.9$. In each example, we compare the stochastic
Galerkin solution to the
Monte-Carlo solution obtained from $10^5$ samples of either the exact solution or a
high-fidelity numerical solution. To simplify the
presentation of the plots, we adopt the legend shown in \autoref{fig:legend}.

\begin{figure}[htbp]
\centering
\includegraphics[width=.9\textwidth,trim=00pt 20pt 0pt 20pt,clip]
{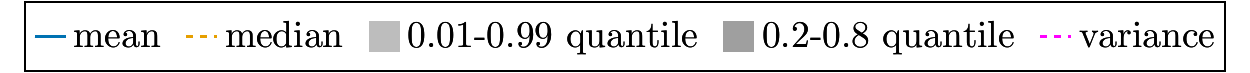}
\caption{The legend used in the plots in the numerical results section.}
\label{fig:legend}
\end{figure}
When the solution is smooth,  we compare the mean, 
median, and variance of the solutions obtained from both SG and Monte-Carlo methods. 
Since the basis is orthonormal, the mean and
variance of the SG solution on each cell are given explicitly by
$\hu_0(t)$ and $\hu_1(t)^2+\dots+\hu_K(t)^2$, respectively.
We use $\hat{E}^{\mathbb{E}}_K$, $\hat{E}^{\mathbb{V}}_K$, and
$\hat{E}^{\mathsf{Med}}_K$ to denote the
difference between the mean, variance, and median of the SG
and Monte-Carlo solutions, respectively. We use $\norm{\cdot}_{1}$ and
$\norm{\cdot}_{\infty}$ to denote the $L^1$ and $L^\infty$ norms over
the spatial domain.

\subsection{Isothermal Euler equations: Smooth solution}
\label{example:isothermal_smooth}

In order to investigate the convergence of the numerical scheme under mesh
refinement, we consider random isothermal Euler equations on
$(-1,1)\times [0,T]\times \Xi$ with non-zero source term such that the exact
solution is given by
\begin{equation}
\label{eqn:exact_solution_isothermal_smooth}
\u(x,t,\xi)=
\begin{bmatrix}
\rho_0+\rho_1 \tanh\left(s(x_0-x+t+\xi)\right)\\
\tanh\left(s(x_0-x+m_0+t+\xi)\right)+\tanh(-s(x_0-x-m_0+t+\xi))
\end{bmatrix}
\end{equation}
where $(\rho_0,\rho_1,s,x_0,m_0)=(0.75,0.25,3,0,0.1)$, and $\xi \sim
\mathcal{U}(-\sigma\sqrt{3},\sigma\sqrt{3})$ with $\sigma=0.1$.

In this example, we use the
AS product to define the SG system \eqref{eqn:isothermal_vector_form}. 
We solve the SG system using the local Lax-Friedrichs method
\eqref{eqn:fully_discrete_FV} from $T=0$ to $T=0.2$ with different mesh sizes
$N=2^{p}, 4\le p\le 13$ and polynomial degrees $1\le K\le 5$. Lastly, we compute the $L^1$
norms of the errors between the mean, variance, and median of the SG solutions and Monte-Carlo solutions,
which we present in \cref{fig:isothermal_smooth_Linfty_errors}.

From the results, we observe that for larger values of $K$, the
errors decrease at a rate of $\mathcal{O}(\Delta x)$ as expected
since the local Lax-Friedrichs method is first-order accurate. We note that for
smaller values of $K$, the errors stagnate as the mesh is refined, which is
expected since the error is dominated by the truncation error in the
polynomial chaos expansion.
This example shows that when a smooth solution is considered and a sufficient
number of polynomial chaos terms are used, the SG method
converges under mesh refinement at the expected rate for the finite-volume
method used.
\begin{figure}[htbp]
\centering
\includegraphics[width=.85\textwidth]{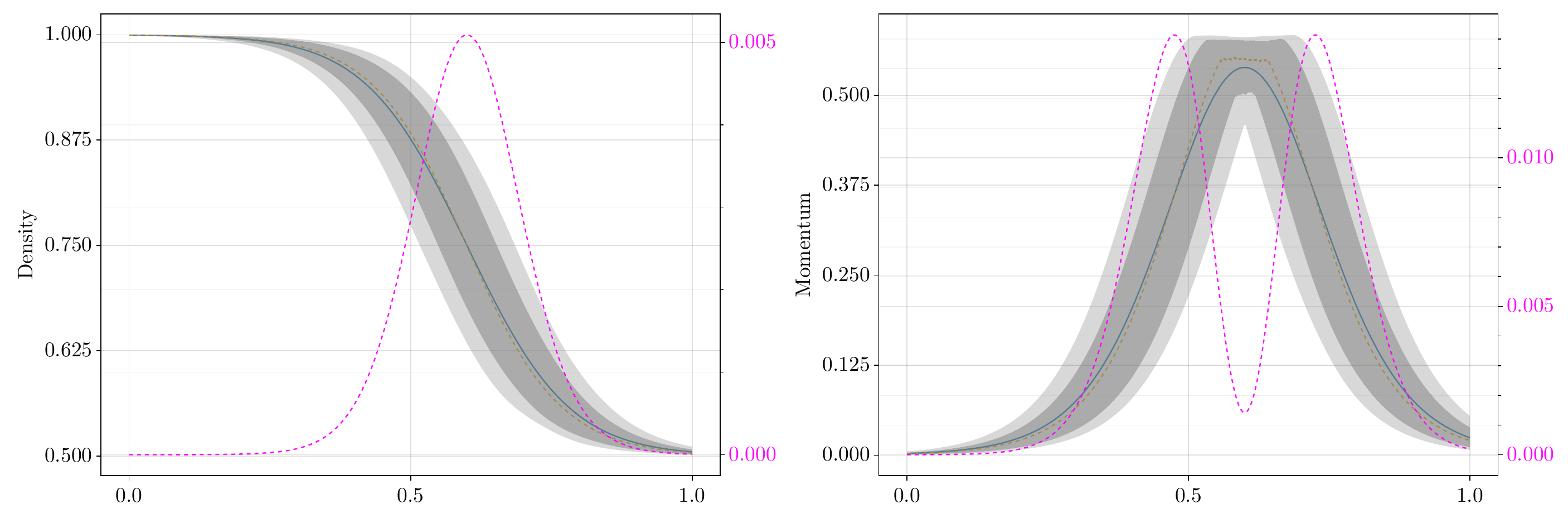}
\caption{The SG  solution to the isothermal Euler equations
given in \eqref{eqn:exact_solution_isothermal_smooth} with $N=2^{13}$ and
$K=3$ at
$t=0.2$.}
\label{fig:reference_solution_isothermal_smooth}
\end{figure}
\begin{figure}[htbp]
\centering
\includegraphics[width=.31\textwidth]{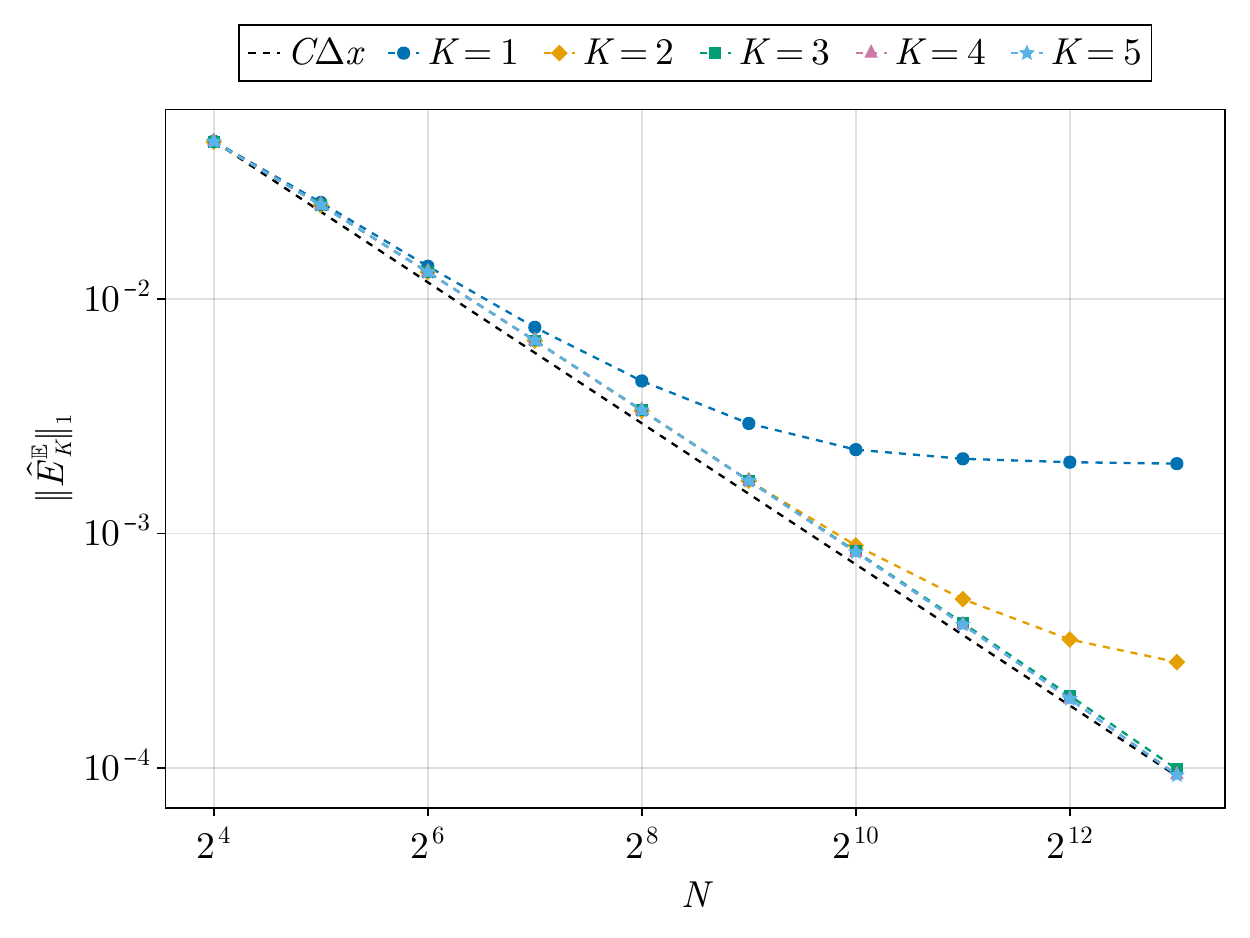}%
\includegraphics[width=.31\textwidth]{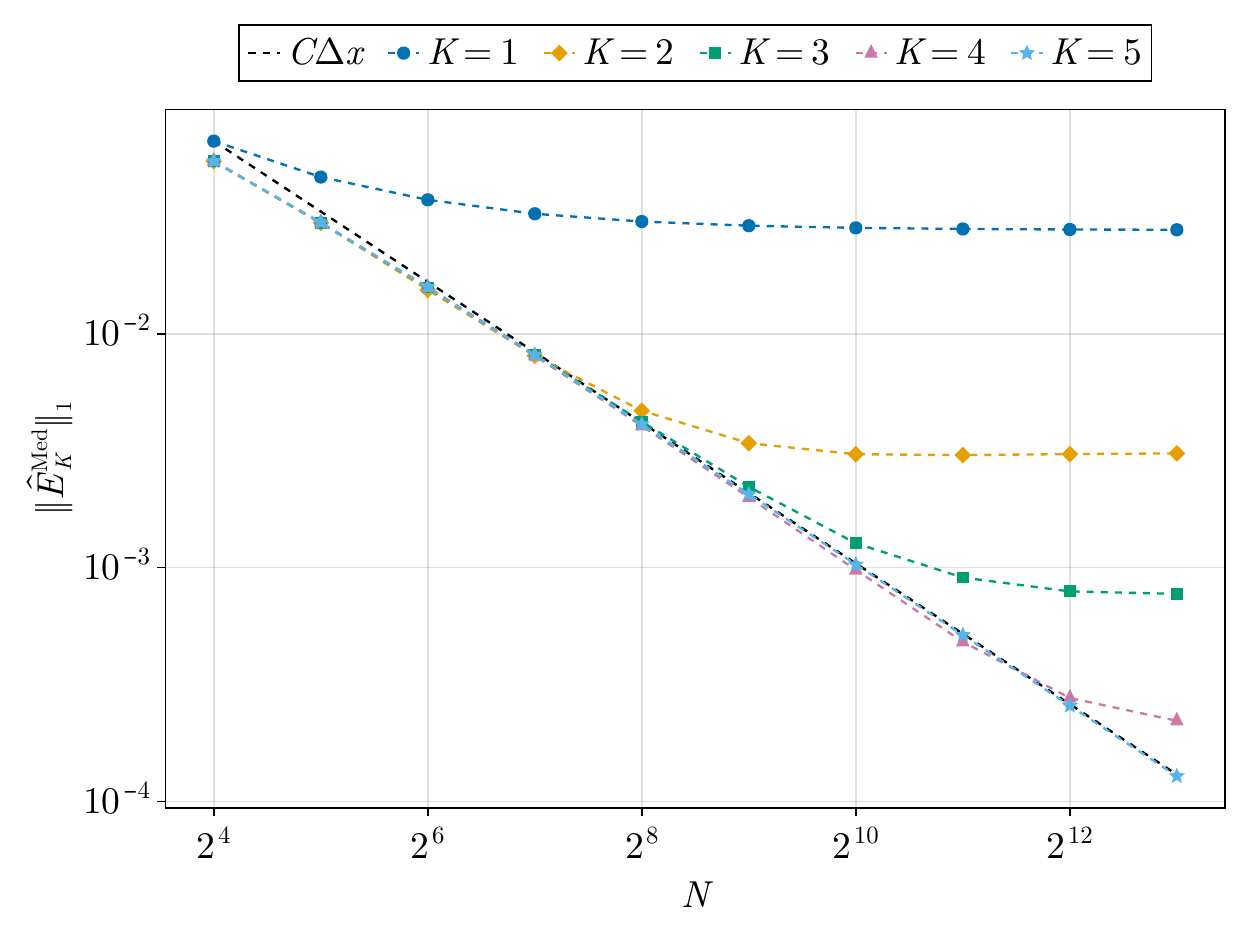}%
\includegraphics[width=.31\textwidth]{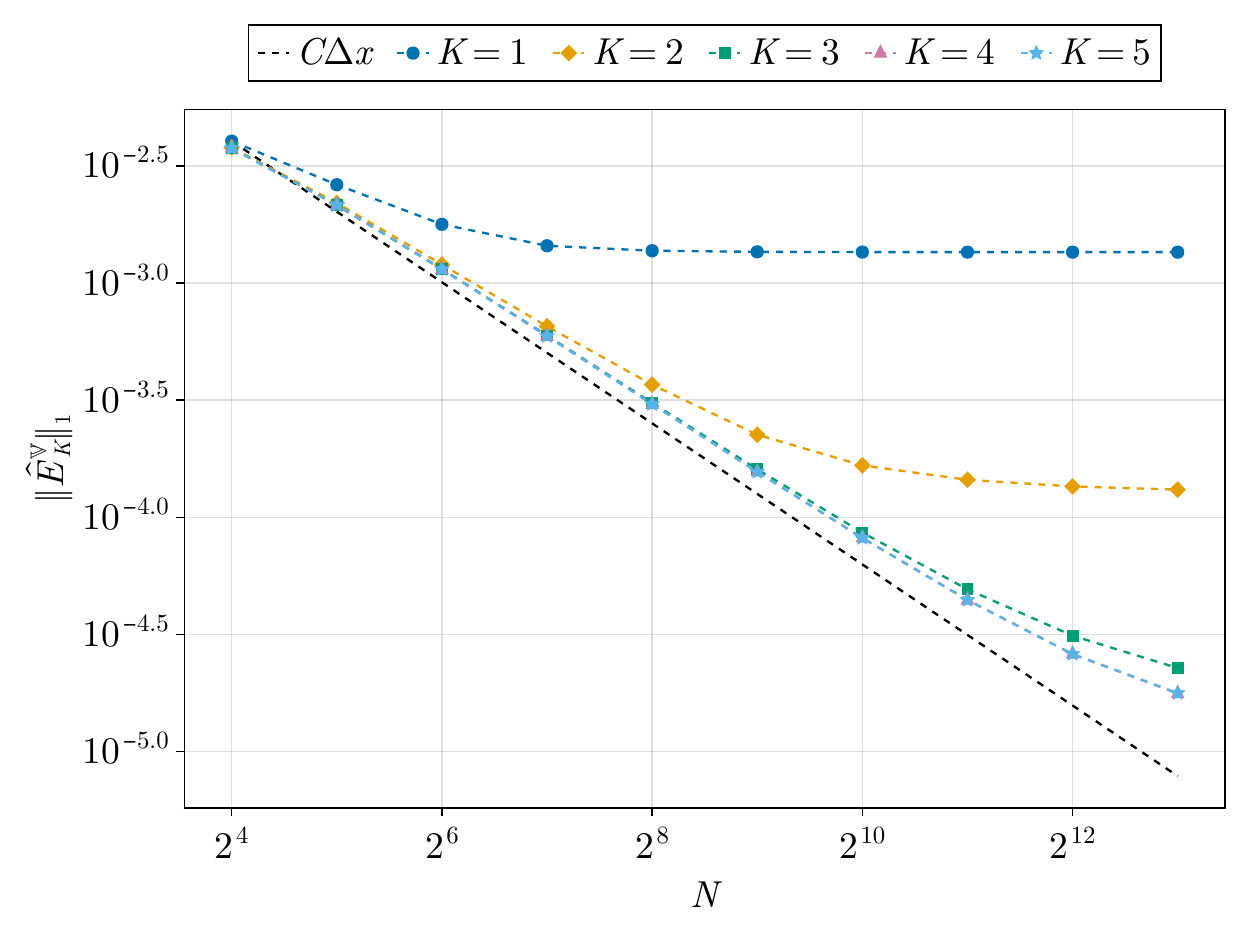}%
\caption{The $L^1$ errors of
the mean (left),  median (center), and variance (right) of the density
obtained from the SG method with different degrees $K$ and
mesh sizes $N$.}
\label{fig:isothermal_smooth_Linfty_errors}
\end{figure}

\subsection{Isothermal Euler equation: Loss of hyperbolicity when using the pseudospectral
  product}
  \label{example:isothermal_loss_of_hyperbolicity}

In this example, we consider the isothermal Euler equations without friction, and an initial
condition chosen such that the pressure is positive for every $\xi$, but the SG system
constructed using the pseudospectral product is not hyperbolic. More precisely,
let  $\Xi=[-1,1]$ and $w\equiv 1/2$ be the uniform distribution on $\Xi$, and let $\{\phi_k\}_{k=0}^3$ be the first four Legendre polynomials orthonormal with respect to $w$. We consider the initial condition
\[
\u(x,0,\xi)=
\begin{cases}
  \sum_{k=0}^3\left( c_k \phi_k(\xi),d_k \phi_k(\xi)\right), & x\in (-1,0),\\
  (0.25,0), & x\in [0,1),
  \end{cases}
  \label{eqn:IC_isothermal_loss_of_hyperbolicity}
\]
where $\mathbf{c}=\{c_k\}_{k=0}^3$ and $\mathbf{d}=\{d_k\}_{k=0}^3$ are
  obtained using the algorithm described in
\cite{jinStudyHyperbolicityKinetic2019} to ensure that the pressure
is positive for
every $\xi$ but the SG system with the pseudospectral product is not hyperbolic.
More specifically, the coefficients $\mathbf{c}$ and $\mathbf{d}$ used in this
example are as follows
\[
\mathbf{c}=
\begin{bmatrix}
1.0\\
0.016660998882728408\\
0.002426705004550297\\
0.02524092143028821
\end{bmatrix}
,\qquad 
\mathbf{d}=
\begin{bmatrix}
0.0\\
0.02867300084262938\\ 
0.02012262860910195\\
0.006875745230701752\\
\end{bmatrix}.
\]
Nevertheless, the SG system constructed using the AS product is
hyperbolic for this initial condition and remains hyperbolic throughout the
timestepping procedure.  The entropy solution of the isothermal Euler equation
with the initial
condition \eqref{eqn:IC_isothermal_loss_of_hyperbolicity} consists of a shock
wave and a rarefaction wave, and an exact solution can be found in the
literature  (cf. \cite{levequeNumericalMethodsConservation1992} or
\cite{gersterStabilizationUncertaintyQuantification2020}).

The numerical results with $K=3$ and $N=2^{14}$ cells, at time $t=0.5$ are shown in
\cref{fig:isothermal_shock_tube_no_loss_sg_solution}. The SG solution captures the main
statistical features of the reference solution, including the behavior near the
shock and rarefaction regions, while the computed states remain in the
admissible hyperbolicity region.
\begin{figure}
\centering
\includegraphics[width=.8\textwidth]{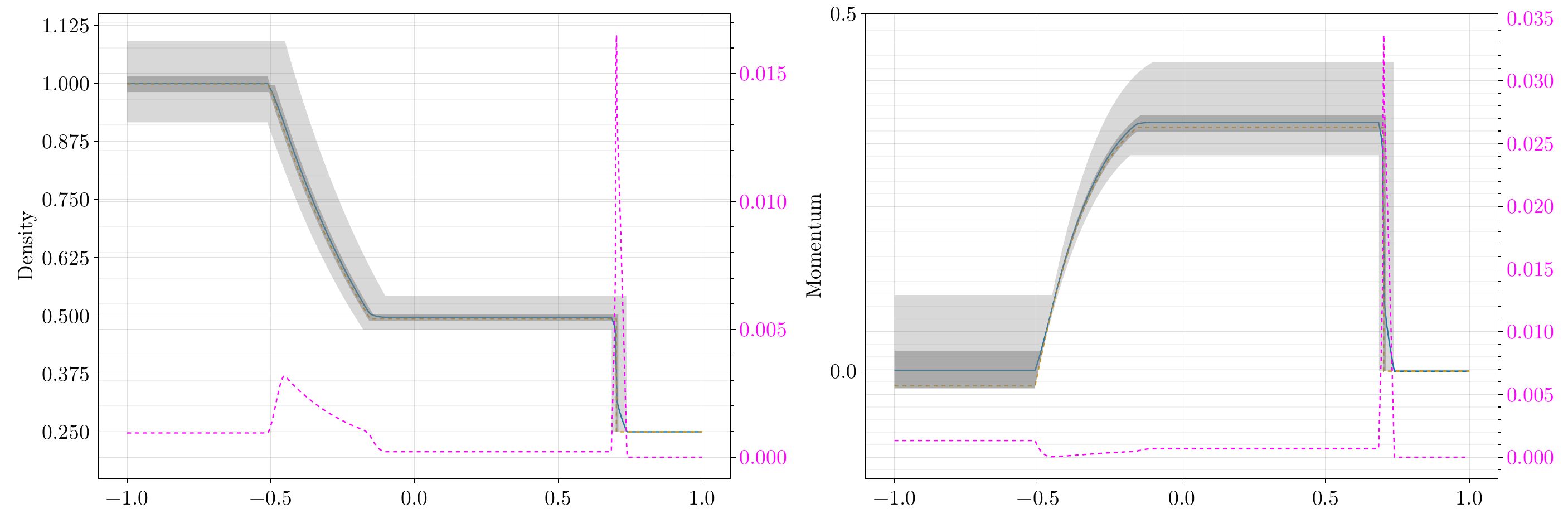}
\includegraphics[width=.8\textwidth]{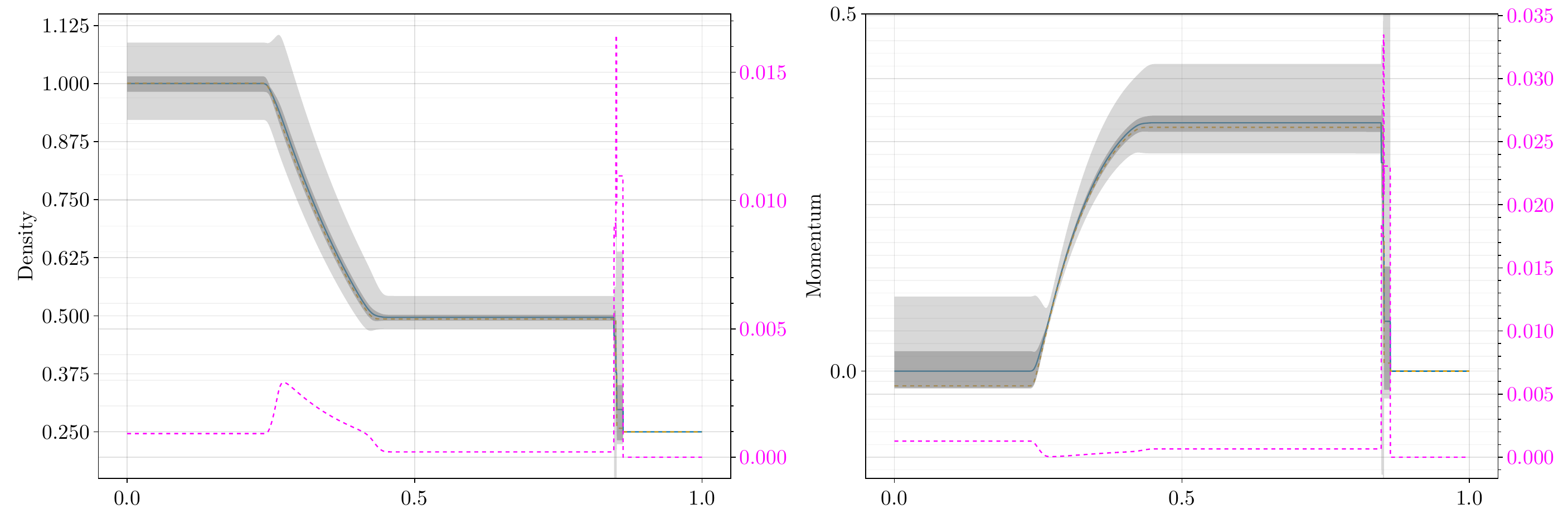}

    \caption{Reference solution (Top) and SG solution (Bottom)
    to the isothermal equation without friction. Here $K=3$ and $N=2^{13}$.
    }
    \label{fig:isothermal_shock_tube_no_loss_sg_solution}
\end{figure}

\subsection{Isothermal Euler equations: Presence of friction}
\label{example:isothermal_pipe}
In this example, we consider the isothermal Euler equations with the
friction source term described in \cref{sec:isothermal_euler}. Here, we consider the
friction coefficient and the diameter of the pipeline to be $f_g=D=1$, and we
assume that initial conditions are given by
\begin{equation}
\u(x,0,\xi)=
\begin{cases}
(\rho_1(\xi),m_1), & x<0,\\
(\rho_2,m_2), & x>0.
\end{cases},\qquad x\in[-1,1],
\label{eqn:IC_isothermal_shock_tube}
\end{equation}
where
$(m_1,\rho_2,m_2)=(0,0.25,0)$, and $\rho_1(\xi) =1+\xi$. Here, $\xi$ follows a
uniform distribution $\mathcal{U}(-\sigma\sqrt{3},\sigma\sqrt{3})$
with $\sigma=0.1$. We compare the solution obtained from the SG method using the
AS product with $K=3$ and $2^{14}$ cells to a reference solution obtained from a Monte-Carlo simulation with
$10^5$ using The Lax-Friedrichs scheme \eqref{eqns:time_discretization}.
The results shown in \cref{fig:isothermal_shock_tube_sg_solution}
indicate that the SG solution captures the behavior of the Monte-Carlo
solution, even at point of high gradients, while the computed states remain in the admissible hyperbolicity
region.

\begin{figure}[htbp]
\centering
\includegraphics[width=0.8\textwidth]{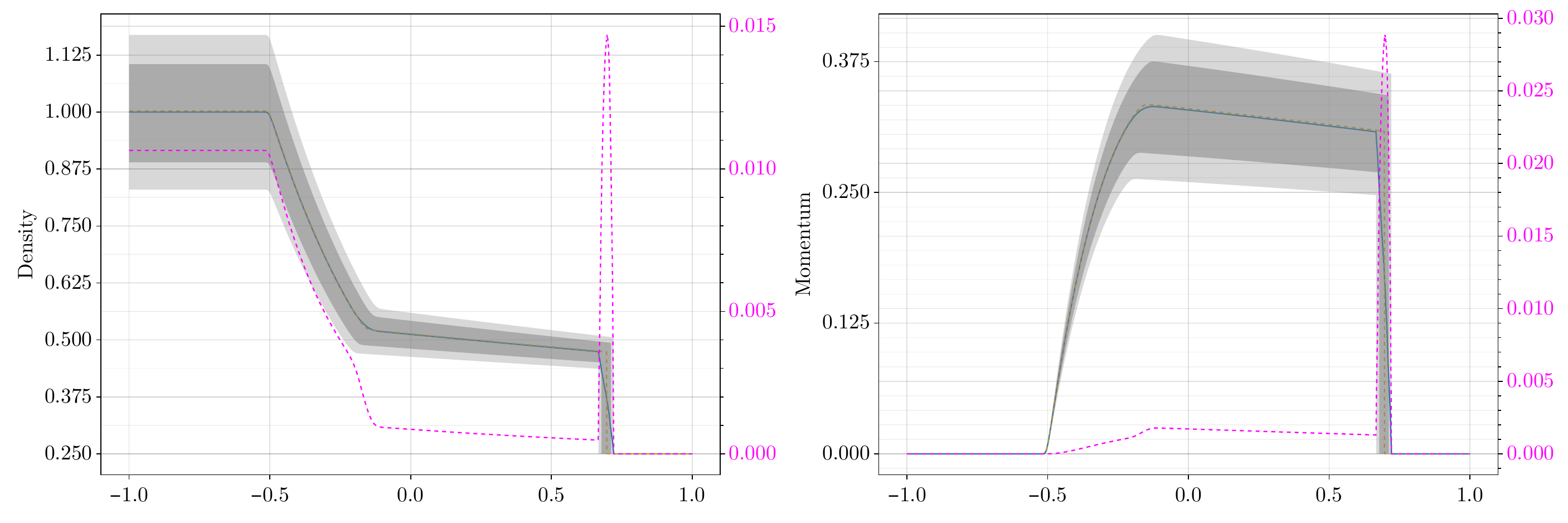}
\includegraphics[width=0.8\textwidth]{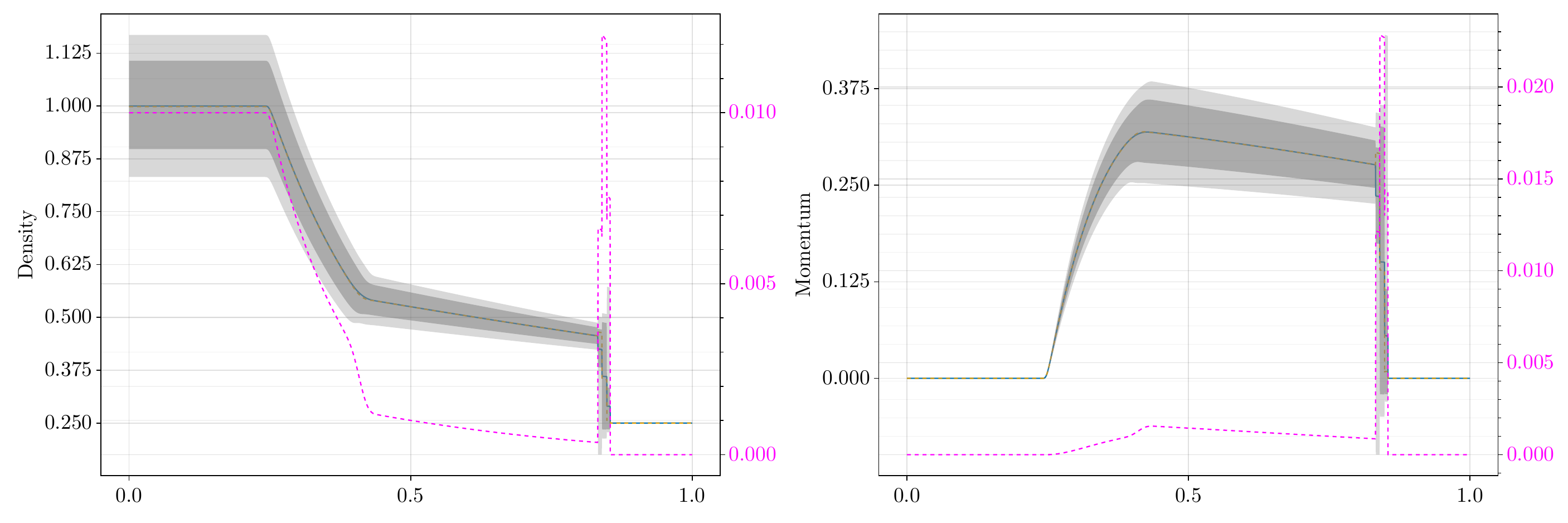}
\caption{The Monte Carlo solution (top) and SG solution (bottom) to the
isothermal Euler equations
with the initial conditions \eqref{eqn:IC_isothermal_shock_tube}
at $t=0.5$ 
using $N=2^{14}$ cells and $K=3$.}
\label{fig:isothermal_shock_tube_sg_solution}
\end{figure}

\subsection{Euler equations: Sod's shock tube}
\label{subsec:euler_sods_shock_tube}
We consider the one-dimensional Euler equations \eqref{eqn:euler_deterministic} on
$[0,1]$ with $\gamma=1.4$ and uncertain Sod-type initial condition
$$\u(x,0,\xi)=
\begin{cases}
(1+\xi,0,2.5), & x<\frac{1}{2},\\
(0.125,0,.25), & x>\frac{1}{2},
\end{cases}$$
where $\xi \sim \mathcal{U}(-\sigma\sqrt{3},\sigma\sqrt{3})$ with
$\sigma=0.1$.
For every $\xi\in \Xi$, the initial state satisfies
$\rho>0$ and $p>0$, and is therefore admissible.
For every fixed $\xi$, the analytical solution to this problem can be found in
\cite{sodSurveySeveralFinite1978}, which we input into a Monte-Carlo
simulation with
$10^5$ samples to obtain the reference solution at $t=0.1$ shown in
\cref{fig:euler_sods_shock_tube_sg_solution} on top.
We solve the SG system \eqref{eqn:Euler_SG_formulation} using
the AS product and the local Lax-Friedrichs method \eqref{eqn:fully_discrete_FV} 
with $N=2^{13}$ cells and $K=3$. The results are shown in
\cref{fig:euler_sods_shock_tube_sg_solution} at the bottom. We
observe that the SG solution captures the mean, median, variance, and confidence intervals
of the reference solution well, even across the shock and contact discontinuity.

\begin{figure}[htbp]
\includegraphics[width=0.99\textwidth]{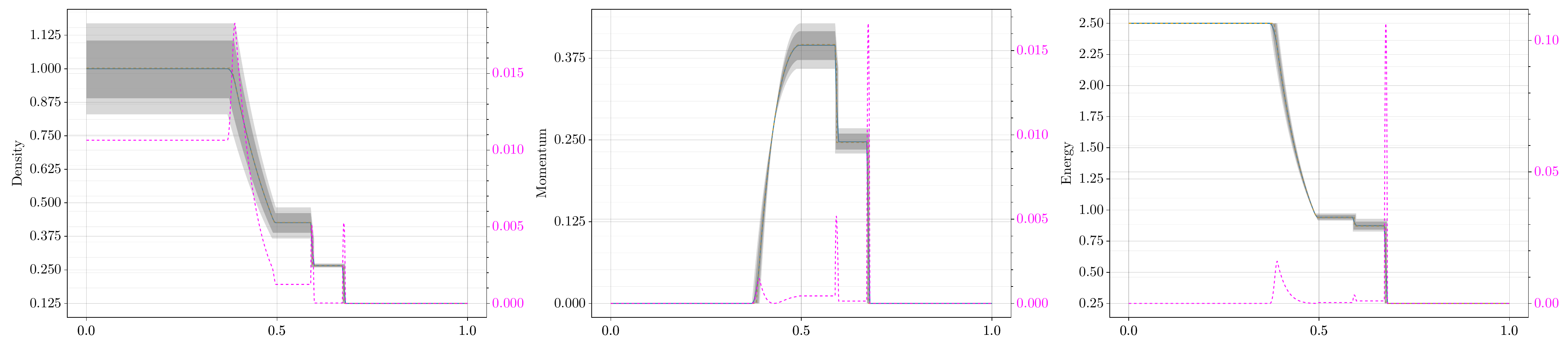}
\includegraphics[width=0.99\textwidth]{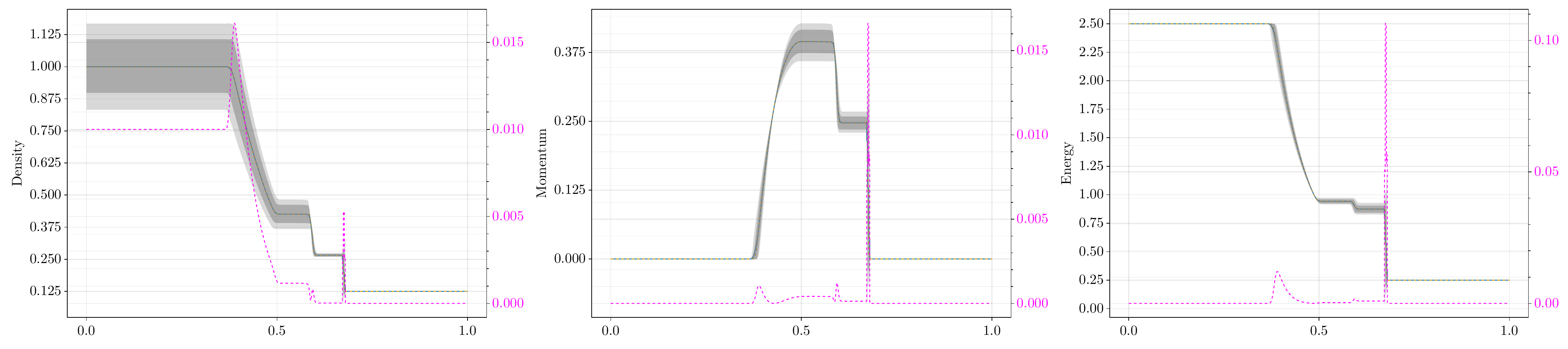}
\caption{Reference solution (top) and SG solution (bottom)
for the Euler Sod shock tube at $t=0.1$ using $N=2^{13}$ cells and $K={3}$.}
\label{fig:euler_sods_shock_tube_sg_solution}
\end{figure}

\subsection{Euler's equations: Two random variables}
\label{subsec:euler_two_d}
Finally, we consider Euler equations \eqref{eqn:euler_deterministic} on
$[0,1]$ with $\gamma=1.4$ and two independent
random variables in the initial data:
\[
\u(x,0,\xi)=
\begin{cases}
(1+\xi_1,0,2.5), & x<\frac{1}{2},\\
(0.125,0,.25+\xi_2), & x>\frac{1}{2},
\end{cases}
\]
where $\xi_i\sim \mathcal{U}(-\sigma\sqrt{3},\sigma\sqrt{3})$ with
$\sigma=0.1$, and the two random variables are independent. 
We solve the SG system
\eqref{eqn:Euler_SG_formulation} using the tensor-product AS product and the local Lax-Friedrichs
method \eqref{eqn:fully_discrete_FV} as in the previous example. Here, the
matrices $\{\M^k\}$ associated with the product are constructed using
\eqref{eqn:multi_dimensional_truncated_product_Mk} from the matrices
associated with the one-dimensional AS product. The results are shown in
\cref{fig:euler_two_d_sg_solution}. Similar to the previous example, we observe
that the SG solution captures the main statistical features of the reference solution.
However, we observe oscillations in the variance of the solution near discontinuity, which may be
due to the Gibbs phenomenon since the solution is discontinuous in the random
space, or due to the choice of the time-stepping scheme.

\begin{figure}
\includegraphics[width=0.99\textwidth]{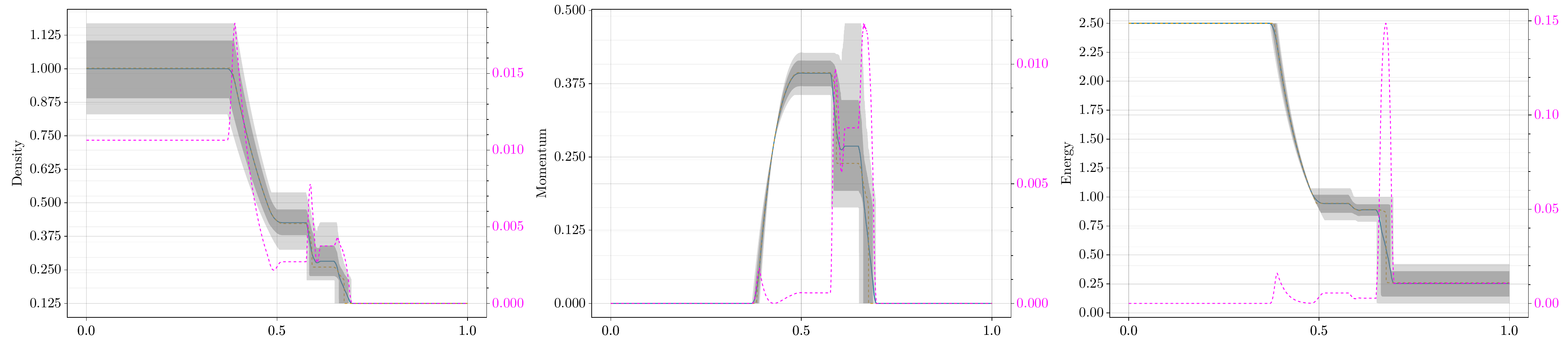}
\includegraphics[width=0.99\textwidth]{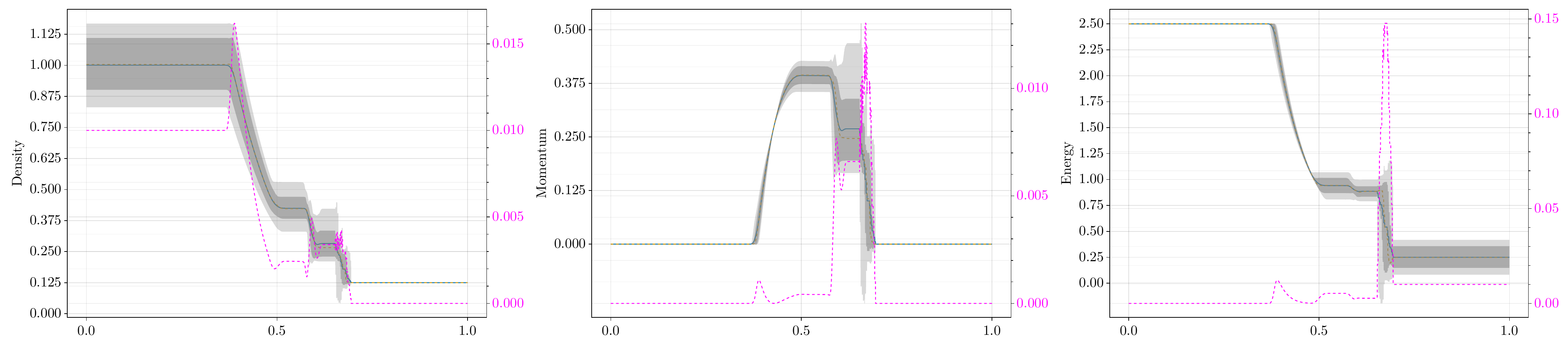}
\caption{Reference solution (top) and SG solution (bottom) for the
Euler equations with two independent random variables
at $t=0.1$ using $N=2^{14}$ cells and tensor-product 
 polynomials of degree $K=3$.}

\label{fig:euler_two_d_sg_solution}
\end{figure}

\section{Conclusion}
\label{sec:conclusion}

We introduced a framework for constructing hyperbolicity-\\preserving stochastic
Galerkin discretizations of nonlinear conservation laws using associative
truncated products on polynomial spaces. The central observation is that
associativity restores the commutativity of the multiplication matrices, which
in turn enables simultaneous diagonalization and provides a mechanism for
preserving hyperbolicity at the SG level.

In one stochastic dimension, we characterized ATPs on
$P_K(\Xi)$ through the single element $\phi_1\ast\phi_K$ and identified several
classes of examples, including collocation products and the associative
symmetric product based on Gaussian quadrature nodes. We also showed that, under
appropriate consistency and stability assumptions, truncated products converge
to the usual product as the polynomial degree increases. For systems with
rational fluxes, we derived sufficient conditions under which the SG 
flux remains hyperbolic, and we applied the framework to the one-dimensional 
isothermal Euler equations and the Euler equations.

The numerical experiments demonstrate that the proposed methods retain
hyperbolicity in the tested regimes while accurately capturing statistical
quantities such as the mean, variance, median, and confidence intervals. In
smooth regimes, the observed convergence under mesh refinement agrees with the
expected rate of the finite-volume discretization, while nonsmooth examples show
that the method remains robust in the presence of shocks.

Several questions remain open. In particular, although the tensor-product
construction extends the framework to multidimensional random variables, the
construction of associative truncated products on sparse or total-degree
polynomial spaces in multiple dimensions remains an important direction for
future work.

\section*{Acknowledgements}
The authors would like to thank Dongbin Xiu for the insightful discussions.
The work of Y. Xing is partially supported by the NSF grant DMS-2309590.	

\bibliographystyle{siam_preamble/siamplain}
\bibliography{refs.bib}

\appendix

\section{Proofs of theorems and lemmas}
\label{sec:appx_proofs}

We first start with the proof of \cref{lem:ps_not_associative}.

\begin{proof}[Proof of \cref{lem:ps_not_associative}]
To prove the claim, it suffices to show that
\begin{equation}
\phi_1\ps \left(\phi_1\ps \phi_K\right)\ne
\left(\phi_1\ps \phi_1\right) \ps \phi_K,\qquad K\ge 2.
\end{equation}

We know that $\phi_1 \phi_K=a_{K+1}^{-1}(\phi_{K+1}-b_{K+1} \phi_K
-c_{K+1}\phi_{K-1})$. Hence, we have \[
\phi_1\ps \phi_K =\frac{-1}{a_{K+1}} \left (b_{K+1} \phi_K
+c_{K+1}\phi_{K-1}\right)
\]
Applying the same idea again, we obtain
\[
\phi_1\ps(\phi_1\ps \phi_K)=
\frac{-b_{K+1}}{a_{K+1}}\phi_1\ps \phi_K
-\frac{c_{K+1}}{a_{K+1}}\phi_1\phi_{K-1}.
\]
On the other hand, if $K\ge 2$, then $\phi_1\ps \phi_1 =
\phi_1^2$. We have
\begin{align*}
\phi_1^2 &\phi_K =\frac{1}{a_{K+1}}\phi_1 \left(\phi_{K+1}-b_{K+1} \phi_K
-c_{K+1}\phi_{K-1}\right) \\
&=\frac{1}{a_{K+1}}\left(\frac{1}{a_{K+2}} \left(\phi_{K+2}-b_{K+1}\phi_{K+1}
-c_{K+2}\phi_K\right)-b_{K+1} \phi_1\phi_K -c_{K+1}\phi_1\phi_{K-1}\right).
\end{align*}
Hence, by ignoring the $\phi_{K+2}$ and $\phi_{K+1}$ terms, we have
\[
(\phi_1\ps\phi_1)\ps \phi_K=
-\frac{c_{K+2}}{a_{K+1}a_{K+2}}\phi_{K}
-\frac{b_{K+1}}{a_{K+1}}\phi_1\ps \phi_K-
\frac{c_{K+1}}{a_{K+1}}\phi_1\phi_{K-1}.
\]
Then, $\phi_1\ps(\phi_1\ps \phi_K)\ne (\phi_1\ps\phi_1)\ps \phi_K$ since
$c_{k},a_k\ne 0$ for all $k\ge 1$ (see \cite{szegőOrthogonalPolynomials1939} for
a proof of the claim $a_k,c_k\ne 0$).
\end{proof}

\begin{proof}[Proof of \cref{lem:associativity_PpPq}]
Let $\hp,\hq,\hr\in \mathbb{R}^{K+1}$, then by the associativity of
$\ast$ we have
\begin{equation}
\G_K\left(\P\left(\P(\hp)\hq\right)\hr \right)
=(p\ast q)\ast r= p\ast (q\ast r)=\G_K\left(\P(\hp)\P(\hq)\hr\right).
\end{equation}
Since $\G_K:\mathbb{R}^{K+1}\to \PK$ is bijective and the equation
above holds for all $\hr\in \mathbb{R}^{K+1}$, we obtain
$\P\left(\P(\hp)\hq\right)= \P(\hp)\P(\hq)$.
The proof of the second equality follows directly from first one:
\[
\P(\hp)\P\left(\P\left(\hp\right)^{-1}\hq\right)
= \P\left(\P(\hp)\P\left(\hp\right)^{-1}\hq\right)
= \P(\hq).
\]
\end{proof}

\end{document}